\documentclass[12pt,a4paper,oneside]{article}
\usepackage{arxiv}

\usepackage[utf8]{inputenc}
\usepackage[numbers]{natbib}
\usepackage{graphicx}
\usepackage{subfig}
\usepackage{placeins} 
\usepackage{float}
\usepackage[linesnumbered,ruled]{algorithm2e}

\usepackage{url}
\usepackage{amsmath}
\usepackage{amsfonts}       
\usepackage{enumitem}
\usepackage{tabularx,booktabs} 
\usepackage{makecell}
\usepackage{multirow, makecell}

\usepackage{pdflscape}

\usepackage{amsthm}
\theoremstyle{definition}
\newtheorem{theorem}{Theorem}[section]

\newtheorem{lemma}[theorem]{Lemma}
\newtheorem{proposition}[theorem]{Proposition}
\newtheorem{remark}[theorem]{Remark}
\newtheorem{definition}{Definition}
\newtheorem{assumption}{Assumption}

\usepackage{ulem}
\usepackage{mathtools}

\DeclareMathOperator*{\argmin}{\arg\!\min}

\newcommand{\norm}[1]{\mathop{}\left\lVert#1\right\rVert}
\newcommand{\scalar}[2]{\mathop{}\left\langle#1,#2\right\rangle}
\newcommand{\eqdef}{\stackrel{\mathrm{def}}{=}}
\newcommand{\dt}{\mathrm{d}t}
\newcommand{\ds}{\mathrm{d}s}
\newcommand{\diff}{\mathrm{d}}
\usepackage[x11names]{xcolor}
\usepackage{hyperref}

\title{Near-optimal Closed-loop Method via Lyapunov Damping for Convex Optimization}

\author{
	\textbf{Severin Maier\qquad \textbf{Camille Castera}$^\ast$}\\
	Department of Mathematics\\
	University of Tübingen\\
	Germany
	\and
	\textbf{Peter Ochs}\\
	Department of Mathematics and Computer Science\\
	Saarland University
	\\
	Germany
}
\newcommand{\labeltext}[2]{%
	\@bsphack
	\def\@currentlabel{#1}{\label{#2}}%
	\@esphack
}
\makeatletter

\begin{document}

	\maketitle

	\begin{abstract}
		We introduce an autonomous system with closed-loop damping for first-order convex optimization. While, to this day, optimal rates of convergence are almost exclusively achieved by non-autonomous methods via open-loop damping (\textit{e.g.}, Nesterov's algorithm), we show that our system, featuring a closed-loop damping, exhibits a rate arbitrarily close to the optimal one. We do so by coupling the damping and the speed of convergence of the system via a well-chosen Lyapunov function.
		By discretizing our system we then derive an algorithm and present numerical experiments supporting our theoretical findings.
	\end{abstract}

	\renewcommand*{\thefootnote}{$^\ast$}
	\footnotetext[1]{Corresponding author: \texttt{camille.castera@protonmail.com}}
	\renewcommand*{\thefootnote}{\arabic{footnote}}
	\setcounter{footnote}{0} 
	
	\section{Introduction}

We consider the unconstrained minimization of a smooth convex real-valued and lower-bounded function $f$ over a Hilbert space $\mathcal{H}$:
	\begin{equation*}
		\min_{x\in \mathcal{H}} f(x).
	\end{equation*}
	One of the most famous algorithms for such optimization problems is Nesterov's Accelerated Gradient method (NAG) \cite{nesterov1983method}, which is known to achieve the optimal rate of convergence for first-order methods on convex functions with Lipschitz-continuous gradient. Among several ways to explain the efficiency of NAG, \citet{NIPS2014} studied the algorithm through the lens of Ordinary Differential Equations (ODEs) and proposed the following model:
	\begin{equation}\label{eq::AVD}
		\tag{\mbox{$ \text{AVD}_{a} $}}
		\ddot{x}(t) + \frac{a}{t}\dot{x}(t) + \nabla f(x(t)) = 0, \qquad \forall t \geq t_0,
	\end{equation}
	where $ a>0 $, $ t_0 \geq 0 $ and $ \nabla f $ denotes the gradient of the smooth real-valued convex function $ f $.
	Here, $ \dot{x}\eqdef\frac{\mathrm{d}x}{\mathrm{d}t}$, respectively $\ddot{x}\eqdef \frac{\mathrm{d}^2x}{\mathrm{d}t^2}$, denotes the first, resp.\ second, time-derivative (or velocity, resp.\ acceleration) of the solution $ x $ of the ODE.
	AVD stands for Asymptotically Vanishing Damping and relates to the coefficient $a/t$.
	NAG is obtained by (non-straightforward) discretization of \eqref{eq::AVD}. Conversely \eqref{eq::AVD} can be seen as NAG with infinitesimal step-sizes. Following \cite{NIPS2014}, many works studied the system \eqref{eq::AVD}, and notably proved that when $a>3$, the function values along a solution (or trajectory) of \eqref{eq::AVD} converge with the asymptotic rate $ o\left(\frac{1}{t^2}\right) $ to the optimal value as $ t \to +\infty $ \cite{NIPS2014,attouch2015fast,may2017asymptotic}. This matches the rate of NAG in the case of  discrete (or iterative) algorithms \citep{nesterov1983method}. We further discuss  additional results for \eqref{eq::AVD} and other possible choices of $a$ later in Section~\ref{sec::relwork}.

	The interest in the connection between ODEs (or dynamical systems) and algorithms comes from the abundance of theory and tools for analyzing ODEs and the insights that they provide for understanding algorithms.
	In particular, continuous-time analyses relying on Lyapunov functions can often be adapted to the discrete setting, see \textit{e.g.}, \cite{attouch2015fast}.
	In this work, we are interested in a special case of the Inertial Damped Gradient (IDG) system:
	\begin{equation}\label{eq::IDG}
		\tag{\mbox{$ \text{IDG}_{\gamma} $}}
		\ddot{x}(t) + \gamma(t)\dot{x}(t) + \nabla f(x(t)) = 0, \qquad \forall t \geq t_0,
	\end{equation}
	where $\gamma$ is a positive function and is called ``damping'' by analogy with mechanics.
	There are two inherently different ways of designing the damping coefficient $\gamma$: the so-called open- and closed-loop manners.
	Optimal convergence rates are usually obtained via open-loop damping, for example, when $\gamma(t)=\frac{a}{t}$ with $a>3$ in \eqref{eq::AVD}. Achieving the same with a closed-loop damping is significantly more challenging as explained hereafter.
	We propose a closed-loop damping that provides \eqref{eq::IDG} with near-optimal rate of convergence. Our special instance of \eqref{eq::IDG} is called \eqref{LD} and is introduced below.
	We first discuss the problem setup in more details.

	\subsection{Problem setting}

	The damping $\frac{a}{t} $ in \eqref{eq::AVD} depends explicitly on the time variable $ t $, making the system non-autonomous.
	We say that a damping with such explicit dependence on the time $t$ is \emph{open-loop}.
	In contrast, a damping $ \gamma $ that does not explicitly depend on $ t $ is called \emph{closed-loop} and makes \eqref{eq::IDG} an autonomous ODE (since the ODE is then independent of $t_0$):
	the feedback for the system, in terms of $\gamma$, may only depend on the state $x$ (and its derivatives) but \emph{not explicitly} on the time $t$.
	For optimization, autonomous ODEs are often preferable over non-autonomous ones, since the dependence on the time $ t $, and hence on the initial time $ t_0 $, is removed. For example, even though the asymptotic rate of convergence of \eqref{eq::AVD} does not depend on $ t_0 $, the trajectory does depend on the choice of $ t_0 $: when choosing a large $ t_0 $, the damping $a/t$ is very small at all time and the ``damping effect'' is almost completely lost. We illustrate such a pathological behavior in Figure \ref{fig::NAGwithoffset}.
	Additionally some important tools for analyzing ODEs only hold (or are simpler to use) for autonomous ODEs, see \textit{e.g.}, the Hartman--Grobman Theorem \cite{hartman1960lemma, grobman1959homeomorphism} or the so-called ODE method \citep{ljung1977analysis,sorin2005}, which formally states that algorithms with vanishing step-sizes asymptotically behave like solutions of corresponding ODEs.

	In this work we therefore tackle the following question:
	\smallskip
	\begin{center}
		\emph{Can one design the damping $\gamma$ in \eqref{eq::IDG} in a closed-loop manner (so as to make the ODE autonomous) while still achieving the optimal convergence rate of \eqref{eq::AVD}?}
	\end{center}
	\smallskip
\begin{figure}[t]%
	\centering
	\begin{minipage}{0.42\linewidth}
		\includegraphics[width=\linewidth]{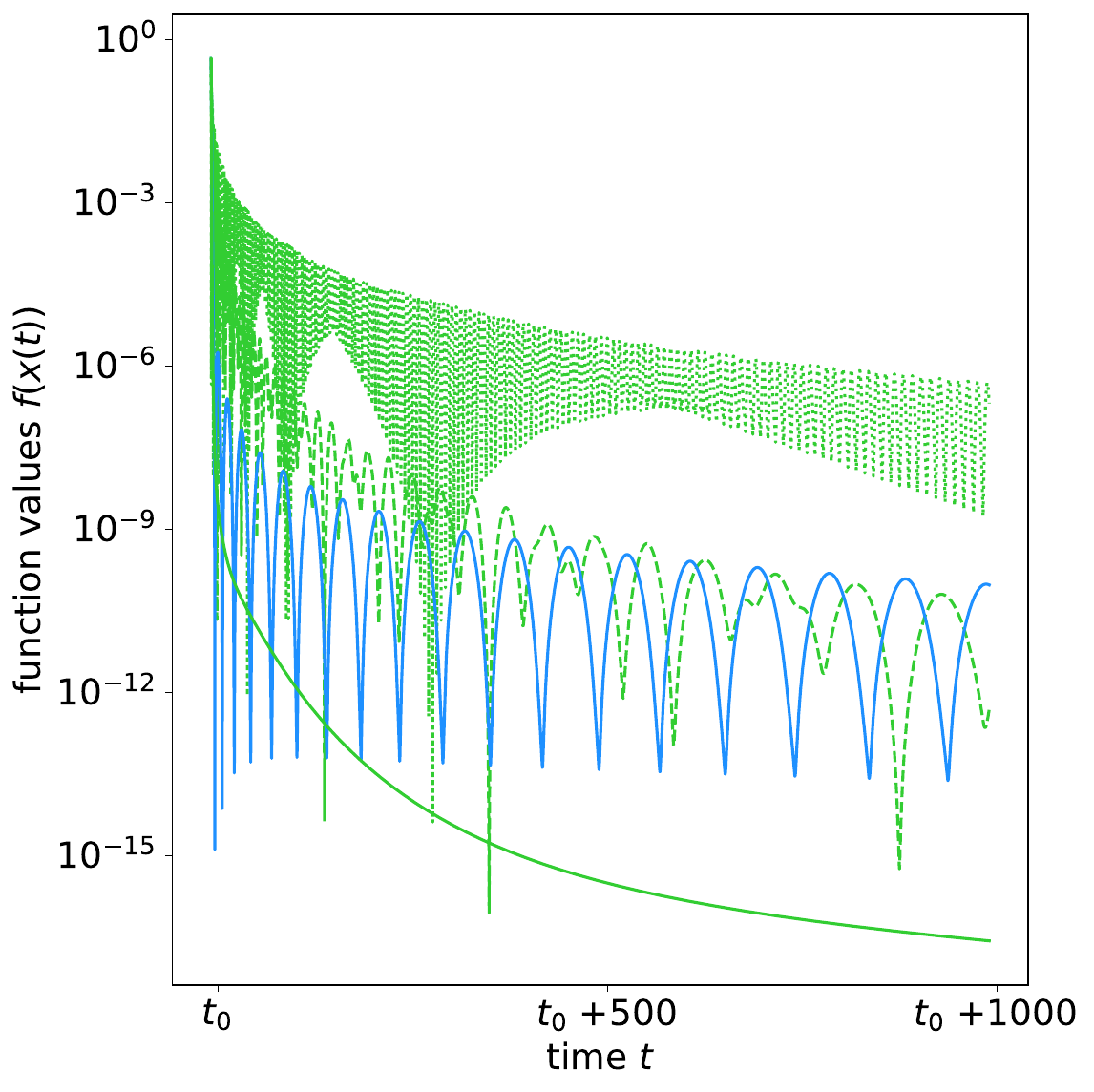}
	\end{minipage}
	\begin{minipage}{0.42\linewidth}
		\includegraphics[width=\linewidth]{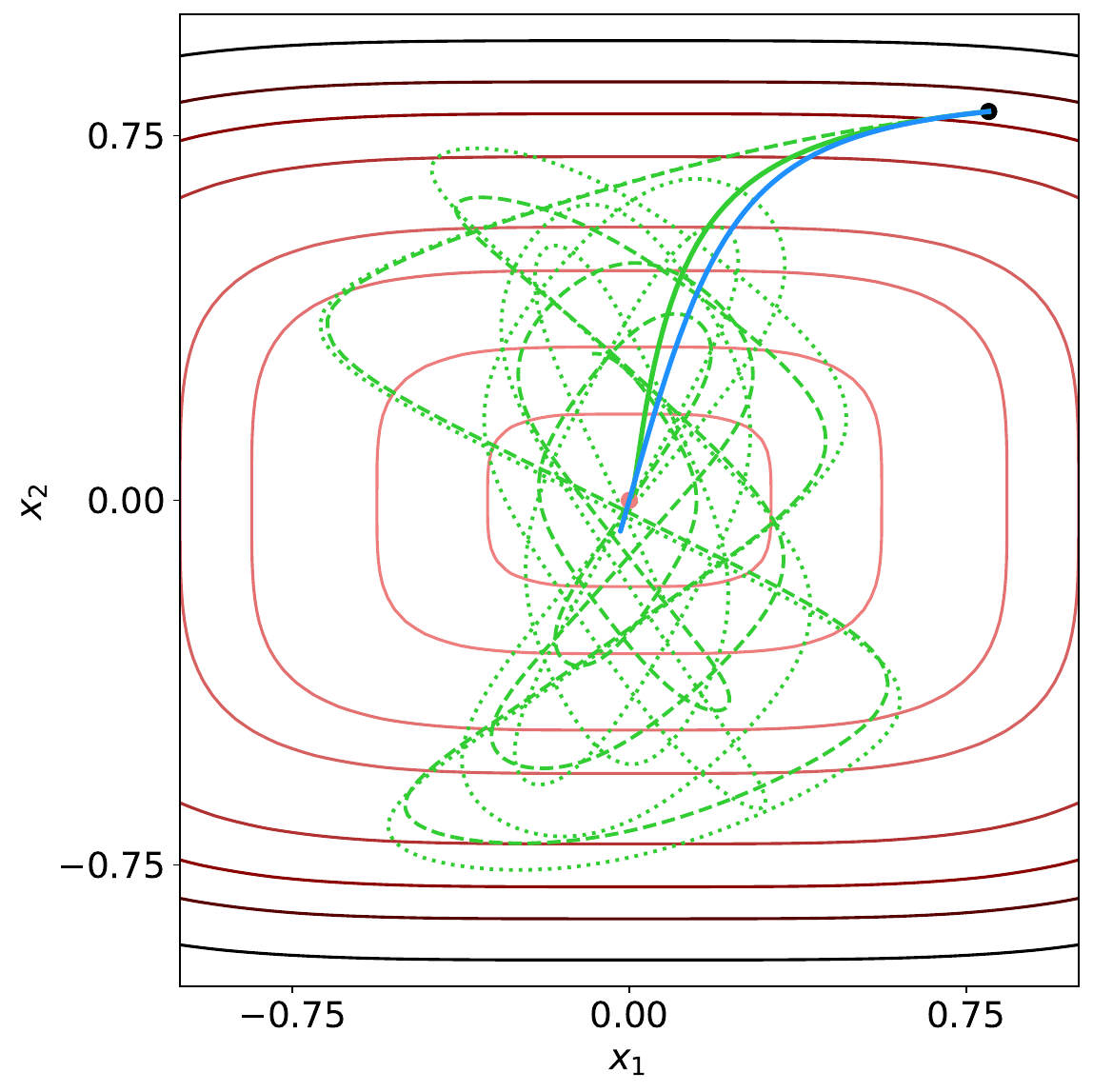}
	\end{minipage}
	\begin{minipage}{0.5\linewidth}
		\includegraphics[width=\linewidth]{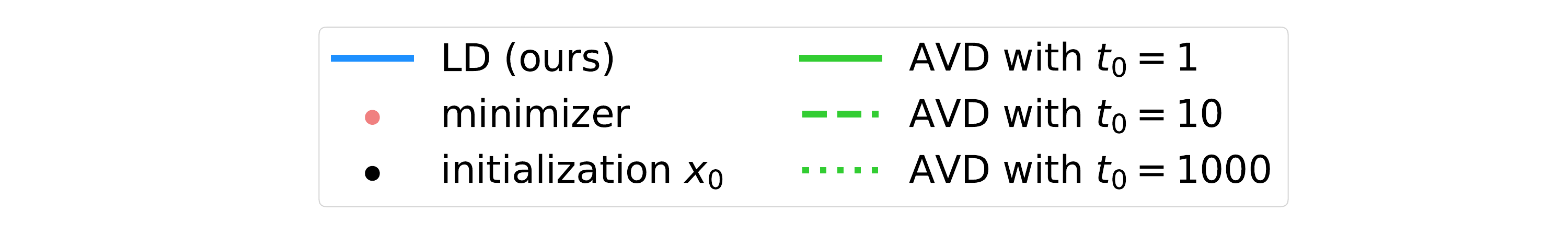}
	\end{minipage}
	\caption{Comparison of our autonomous system \eqref{LD} and the non-autonomous one \eqref{eq::AVD} for different initial times $ t_0 $ on the 2D function $ f(x_1,x_2) = x_1^4+0.1x_2^4 $. The left plot shows the evolution of the function values over time. The right plot shows the trajectories in the space $(x_1,x_2)$. Different initial times heavily affect the solution of \eqref{eq::AVD}, but not \eqref{LD}. We approximated the solutions using NAG for \eqref{eq::AVD} and LYDIA (see Algorithm~\ref{algo::LD}) for \eqref{LD}, both with very small step-sizes.
		\label{fig::NAGwithoffset}}
\end{figure}

	This question hides a circular argument that makes it hard to solve. Indeed, according to \citet{cabot2009long}, a sufficient condition on $\gamma$ for the convergence of the function values is that $\int_{t_0}^{+\infty} \gamma(t)\dt = +\infty $.
	The straightforward choice is then the constant damping $ \gamma(t) = a > 0 $, which in \eqref{eq::IDG} yields the Heavy Ball with Friction (HBF) algorithm \citep{polyak1964}, known to be sub-optimal for convex functions.
	We therefore seek a closed-loop damping that converges to zero, but not ``too fast''.
	In the open-loop setting, the natural choice is $ \gamma(t)=\frac{a}{t}, \ a > 0 $ because this is the fastest polynomial decay that converges to zero while being non-integrable.
	Further \cite{attouch2017timedependentviscosity} showed that any damping of the form $ \frac{a}{t^\beta} $ with $ \beta \in (0,1) $ yields a sub-optimal rate.
	Therefore we formally seek to design a closed-loop damping $\gamma$ that behaves like $ \frac{1}{t} $.
	Because it cannot depend directly on the variable $t$, our closed-loop damping $\gamma$ must be built from other quantities of the system that may converge to zero, for example $\norm{\dot{x}(t)}$. \emph{However, the rate at which such quantities converge to zero depend themselves on the choice of $\gamma$}.
	To escape this loop of thought, we design $\gamma$ using quantities that behave asymptotically like $ \frac{1}{t} $ in the case of \eqref{eq::AVD} and investigate whether this choice of $\gamma$ still gives fast convergence rates.  This idea takes inspiration from \citep{attouch2022fast} which used $\norm{\dot{x}(t)}$ as damping, with a key difference that we now explain.


	\subsection{Our contribution}

	In the case of \eqref{eq::AVD} with $a>3$, it is known that under standard assumptions the quantity $E(t) \eqdef f(x(t))- f^\star + \frac{1}{2}\norm{\dot{x}(t)}^2 $ (where $ f^\star\eqdef\min_{\mathcal H}f$) is such that $E(t)=o\left(\frac{1}{t^2}\right)$  \citep{may2017asymptotic}. Therefore, following the aforementioned intuition, we propose a closed-loop damping defined for all $t\geq t_0$ by $ \gamma(t) \eqdef \sqrt{E(t)} $.
	Our version of \eqref{eq::IDG} then reads:
	\begin{equation}\label{LD}
			\tag{LD}
			\ddot{x}(t) + \sqrt{E(t)} \dot{x}(t) + \nabla f(x(t)) = 0, \qquad \forall t \geq t_0,
	\end{equation}
	and is called \eqref{LD} for \emph{Lyapunov Damping}. It will indeed turn out that $ E $ is non-increasing, making it a ``Lyapunov function'' for \eqref{LD}.
	Note that \eqref{LD} assumes the availability of the optimal value $f^\star$, we further discuss this later in Remark~\ref{rem::fstar}.
	Our main result is that our closed-loop damping system \eqref{LD} yields the following rate of convergence.
	\begin{theorem}\label{thm::fastratedelta}
		Assume that $ f $ is a continuously differentiable convex function and that $ \argmin_{\mathcal{H}} f \neq \emptyset $. Then, for any bounded solution $ x $ of \eqref{LD} and for any $ \delta > 0 $,
		\begin{equation*}
			f(x(t))-f^\star = o\left(\frac{1}{t^{2-\delta}}\right).
		\end{equation*}
	\end{theorem}
	This means that the rate of our method is arbitrarily close to $o\left(\frac{1}{t^2}\right)$, the optimal one achieved by \eqref{eq::AVD} with $ a > 3 $.
	Also, note that \eqref{LD} does not make use of the hyper-parameter $a$, whose choice is crucial in \eqref{eq::AVD}.

	\subsection{Organization}
	In Section \ref{sec::relwork} we review related work regarding ODEs for optimization, closed-loop damping, and known results in the open-loop case. In Section \ref{sec::prelandexistence} we make the setting  precise and show the existence of solutions to \eqref{LD}. Section~\ref{sec::main} is devoted to showing our main result. In Section \ref{sec::exp} we derive a practical algorithm from \eqref{LD} and use it to perform numerical experiments. We finish by drawing conclusions and further discussing our results.


	\section{Related Work}\label{sec::relwork}
	\paragraph{ODEs for optimization}
	There is a long line of work in exploiting the interplay between ODEs and optimization algorithms, going back, at least, to the work of Polyak's \cite{polyak1964} Heavy Ball with Friction (HBF) method for acceleration. As previously stated, \citet{NIPS2014} linked NAG \cite{nesterov1983method} to the differential equation \eqref{eq::AVD}, hence providing a new view on the heavily used, yet not perfectly understood algorithm. NAG is however obtained via a non-straightforward discretization of \eqref{eq::AVD} since the gradient of $f$ is evaluated at an ``extrapolated point''. Recently, \citet{alecsa2021implicit} proposed a model with ``implicit Hessian'', whose Euler explicit discretization yields NAG. Higher-order ODEs have also been introduced to better understand NAG: \citet{attouch2016hessian} proposed a model with ``Hessian damping'' based on \citep{alvarez2002}, while \citet{shi2021resolution} similarly considered higher-resolution ODEs allowing to better distinguish NAG from other \eqref{eq::IDG} systems.

	\paragraph{Closed-loop damping}

	The closest work to ours is \cite{attouch2022fast}, which proposed a closed-loop damping of the form $ \gamma = \norm{\dot{x}}^p $ for several values of $ p \in \mathbb{R} $.
	Our work builds on \cite{attouch2022fast} since our damping $\sqrt{E}$ also involves $\norm{\dot{x}}$. Yet, \citet{attouch2022fast} could provide counterexamples for which their systems do not achieve near-optimal rates, unlike ours. A key difference is that our damping $\sqrt{E}$ is non-increasing, while $\norm{\dot{x}}$ may oscillate heavily.
	They nonetheless derived rates under additional assumptions (\textit{e.g.} strong convexity or the Kurdyka-{\L}ojasiewicz (KL) property).

	Another line of work regarding closed-loop damping consists in deriving them from a feedback loop that follows an algebraic equation. This was initiated by \citep{attouch2016dynamic} for proximal Newton's method and then used for first-order optimization by \cite{lin2022control} then generalized by \cite{attouch2023timescaling}. These works rely on a so-called ``time-rescaling'' technique, which formally amounts to accelerating convergence by making time go faster. This allows obtaining optimal convergence rates in the ODE setting, with the drawback that this acceleration is \emph{artificial}: it cannot be transferred to algorithms via explicit discretization.\footnote{Discretizing time-rescaling amounts to use larger step-sizes which is not possible without causing numerical instability issues.} By combining this approach with ``time-averaging'', \cite{lin2022control,attouch2023timescaling} could further obtain ``non-artificial'' acceleration and optimal rates from the ODE perspective. They then derived practical algorithms that are faster than the sub-optimal $O(\frac{1}{k})$ of gradient descent but still sub-optimal compared to NAG. Furthermore \citep{attouch2023timescaling} must resort to using implicit discretization (via proximal operators) which is expensive in general. In contrast our approach does not involve time-rescaling, which makes it possible to use explicit Euler discretization ensuring numerical stability under standard assumptions (see Theorem~\ref{thm::LYDIAconv}).

	Regarding other work involving closed-loop damping,
	\citet{adly2006} proposed to replace the term $ \gamma(t)\dot{x}(t) $ in \eqref{eq::IDG} by a ``non-smooth'' potential. Their non-smooth potential provides the solutions of their differential inclusion with finite length and convergence to a point very close to a minimizer of $f$ (if such a point exists). However, since the solution does not converge to the minimal value, it does not have the optimal rate that we seek.

	Our damping makes use of the optimality gap $ f(x(t)) - f^\star $ (see Remark~\ref{rem::fstar} for further discussion). This idea is not new and is used, for example, in the Polyak step-size $ \frac{f(x_k)-f^\star}{\norm{\nabla f(x_k)}^2}$ \cite{polyak1969minimization}. In our system \eqref{LD} the optimality gap is rather used for designing the damping $\gamma$ rather than the step-size.

	\paragraph{Open-loop damping and proof techniques}
	The convergence properties of \eqref{eq::AVD} (and more generally \eqref{eq::IDG} in the open-loop setting) have been intensively studied. First, for \eqref{eq::AVD} with $ a > 3 $, following \citep{NIPS2014}, the convergence rate $ o\left(\frac{1}{t^2}\right) $ for function values and the convergence of the trajectories were proved in \cite{may2017asymptotic, attouch2015fast,chambolle2015convergence}. In the setting $a<3$,  \cite{aujol2017subcritical,attouch2019subcritical} derived the sub-optimal rate $ \mathcal{O}\left(\frac{1}{t^{\frac{2a}{3}}} \right) $. For the critical value $ a = 3 $ only the rate $ \mathcal{O}\left(\frac{1}{t^2}\right) $ is known and whether this rate can be improved to $ o\left(\frac{1}{t^2}\right) $ and whether the trajectories converge remains open.
	Rates have also been derived under additional assumptions such as the KL property \citep{aujol2019optimal}.

	Regarding \eqref{eq::IDG}, \citet{attouch2017timedependentviscosity} developed general conditions for the convergence of the values and of the trajectories, which unify several results mentioned above. All these results have in common that they rely on the analysis of Lyapunov functions, like $E$ previously introduced. We refer to \cite{wibisono2016variational, wilson2021lyapunov} for more details on Lyapunov analyses.
 	The proof of our main result takes inspiration from those in \cite{cabot2012hyperbolic} where sufficient conditions to derive optimal rates are provided. They must nonetheless be significantly adapted since some of the conditions in \cite{cabot2012hyperbolic} do not hold generally for closed-loop dampings like ours. We replace them by using a specific property of our system (see Lemma \ref{lem::derEanddamping} hereafter). A more detailed comparison follows our main analysis in Remark \ref{rem::comparisonhyperbolic}.


	\section{Preliminaries and Existence of Solutions}\label{sec::prelandexistence}
	Throughout the paper we fix a real Hilbert space $\mathcal{H}$ with inner product $ \scalar{\cdot}{\cdot} $ and induced norm $ \norm{\cdot} $. \\
	We make the following assumptions on the function $ f $.
	\begin{assumption}\label{ass::f}
		The function $ f:\mathcal{H} \to \mathbb{R} $ is
		\begin{enumerate}[label=(\roman*),noitemsep,topsep=0pt]
			\itemsep=4pt
			\item convex and continuously differentiable with locally Lipschitz-continuous gradient $ \nabla f $;
			\item bounded from below by $ f^\star \eqdef \inf_{\mathcal{H}} f. $
		\end{enumerate}
	\end{assumption}
	We also fix an initial time $ t_0 \geq 0 $ and initial conditions $ x(t_0) = x_0 \in \mathcal{H} $, and $ \dot{x}(t_0) = \dot{x}_0 \in \mathcal{H}$.
	\begin{definition}
		A function $ x: [t_0,+\infty[ \to \mathcal{H}$, which is twice continuously differentiable on $ ]t_0,+\infty[ $ and continuously differentiable on $ [t_0,+\infty[ $, is called a (global) \emph{solution} or \emph{trajectory} to \eqref{eq::IDG}, resp.\ \eqref{LD}, if it satisfies \eqref{eq::IDG}, resp.\ \eqref{LD}, for all $t> t_0$ and satisfies the initial conditions previously mentioned.
	\end{definition}

	Given the setting above, we can ensure the existence and uniqueness of the solutions of \eqref{LD}.
	\begin{theorem}\label{thm::existence}
		Under Assumption \ref{ass::f}, there exists a unique solution $ x $ to \eqref{LD} with initial conditions $\left(x_0,\dot{x}_0\right)\in \mathcal{H} \times \mathcal{H} $ and initial time $ t_0 \geq 0 $.
	\end{theorem}
	The proof relies on the Picard--Lindelöf Theorem and is postponed to Appendix \ref{app::existence}.
	\begin{remark}
	The local Lipschitz-continuity of $\nabla f$ is only required to guarantee the existence and uniqueness of the solutions on $ [t_0,+\infty[ $ (see below), but is not used elsewhere in our analysis.
	\end{remark}

	We finally recall a special case of the Landau notation for asymptotic comparison that we use heavily in the sequel.
	\begin{definition}\label{def::biglittleo}
		For any non-negative function $ g\colon\mathbb{R}\to\mathbb{R} $ and $ \forall \alpha \geq 0 $,
		\begin{align*}
			g(t) = o\left(\frac{1}{t^\alpha}\right) &\iff \lim_{t \to +\infty} t^\alpha g(t) = 0.
		\end{align*}
	\end{definition}


	\section{Main Result: Convergence Rates for \texorpdfstring{\eqref{LD}}{LD}}\label{sec::main}
	This section is devoted to proving our main result Theorem \ref{thm::fastratedelta}, which states that the solution of \eqref{LD} achieves a convergence rate that is arbitrarily close to the optimal one for convex functions. First we show that $ E $ is indeed a Lyapunov function for the system \eqref{LD}.
	\subsection{The Lyapunov function of \texorpdfstring{\eqref{LD}}{LD}}
	Throughout what follows, let $x$ be the solution of \eqref{LD} with the initial conditions stated in Section~\ref{sec::prelandexistence}.
	Recall that the damping coefficient in \eqref{LD} is $\sqrt{E(t)}$ where for all $t\geq t_0$, $E(t)= f(x(t))-f^\star + \frac{1}{2}\norm{\dot{x}(t)}^2$. We first show identities that are specific to \eqref{LD} and that play a crucial role in what follows.
	\begin{lemma}\label{lem::derEanddamping}
		Under Assumption \ref{ass::f} the solution $ x $ of \eqref{LD} is such that $ E $ is continuously differentiable for all $ t \geq t_0 $ and
		\begin{equation}\label{eq::derE}
			\frac{\diff E(t)}{\dt} = - \sqrt{E(t)}\norm{\dot{x}(t)}^2,
		\end{equation}
		so in particular $ E $ is non-increasing. Furthermore, for all $ t \geq t_0 $ it holds that
		\begin{equation}\label{eq::dersqrtE}
			\frac{\diff \sqrt{E(t)}}{\dt} = - \frac{1}{2} \norm{\dot{x}(t)}^2, \quad \text{or equivalently,} \quad \int_{t_0}^{t}\norm{\dot{x}(s)}^2\ds = 2\sqrt{E(t_0)}-2\sqrt{E(t)}.
		\end{equation}
		\begin{proof}
			We apply the chain rule and use the fact that $x$ solves \eqref{LD} to obtain:
				\begin{multline}
					\frac{\diff E(t)}{\dt}  = \scalar{\nabla f(x(t))}{\dot{x}(t)} + \scalar{\ddot{x}(t)}{\dot{x}(t)} \\
					= \scalar{\nabla f(x(t)) - \sqrt{E(t)}\dot{x}(t) - \nabla f(x(t))}{\dot{x}(t)}
					= - \sqrt{E(t)}\norm{\dot{x}(t)}^2,
				\end{multline}
			which proves \eqref{eq::derE}.
			As for the second part, the chain rule and \eqref{eq::derE} imply
			\begin{equation*}
				\frac{\diff\sqrt{E(t)}}{\dt} = \frac{1}{2\sqrt{E(t)}}\frac{\diff E(t)}{\dt} = - \frac{1}{2\sqrt{E(t)}} \sqrt{E(t)}\norm{\dot{x}(t)}^2 = - \frac{1}{2} \norm{\dot{x}(t)}^2.
			\end{equation*}
			Finally, by the Fundamental Theorem of Calculus and the continuity of $ \dot{x} $ this is equivalent to
			\begin{equation*}
				\int_{t_0}^{t}\norm{\dot{x}(s)}^2\ds = 2\sqrt{E(t_0)}-2\sqrt{E(t)}. 
			\end{equation*}
		\end{proof}
	\end{lemma}
	The function $ E $ describes a quantity which is non-increasing along the trajectory $ x $, and note that $ E(t) = 0 $ if, and only if, $ f(x(t)) = f^\star $ and $ \dot{x}(t) = 0 $. Such a function is called \emph{a Lyapunov function} for the system \eqref{LD}.


	\subsection{Preliminary convergence results}
	We make the following assumptions, which are consistent with those in Theorem \ref{thm::fastratedelta}.
	\begin{assumption}\label{ass::minsol}We assume that
		\begin{enumerate}[label=(\roman*),noitemsep,topsep=0pt]
			\itemsep=4pt
			\item $\argmin_{\mathcal{H}} f \neq \emptyset$;
			\item the solution $x$ of \eqref{LD} is uniformly bounded on $[t_0,+\infty[$.\label{bounded}
		\end{enumerate}
	\end{assumption}
	\begin{remark}
		Assumption~\ref{ass::minsol}-\ref{bounded} holds, for example, when $f$ is coercive.
	\end{remark}

	We begin our analysis by showing that the trajectory $ x $ minimizes the function $ f $. We do so by showing that $ E(t) $ tends to zero as $ t \to +\infty $.
	\begin{theorem}\label{thm::convto0}
		Under Assumptions~\ref{ass::f} and~\ref{ass::minsol}, $ E(t) $ converges to zero as $ t \to +\infty $. This implies in particular that $ \displaystyle f(x(t)) \xrightarrow[t \to + \infty]{} f^\star $ and $\displaystyle \norm{\dot{x}(t)} \xrightarrow[t \to +\infty]{} 0 $.
	\end{theorem}
	We make use of the following classical result to prove Theorem \ref{thm::convto0}.
	\begin{lemma}\label{lem::convto0}
		Let $ g $ be a non-negative continuous function on $ [t_0,+\infty[ $ such that $\int_{t_0}^{+\infty}g(t)\dt$ is finite, then either $\displaystyle\lim_{t\to+\infty}g(t)$ does not exist or $\displaystyle \lim_{t\to+\infty}g(t) = 0$.
	\end{lemma}
	The proof of Lemma \ref{lem::convto0} is postponed to Appendix \ref{app::proofsoflemmata}. We now prove  Theorem~\ref{thm::convto0}.
	\begin{proof}[Proof of Theorem \ref{thm::convto0}]
		Let $ z \in \argmin_{\mathcal{H}}f $, and for all $ t \geq t_0 $ define the so-called ``anchor function'' $ h_z(t) \eqdef \frac{1}{2}\norm{x(t) - z}^2 $. Then $ h_z $ is twice differentiable for all $ t \geq t_0 $ and we have:
		\begin{align}
			\dot{h_z}(t) &= \scalar{x(t)-z}{\dot{x}(t)}. \label{eq::hztdot}\\
			\begin{split}
			\ddot{h_z}(t) &= \norm{\dot{x}(t)}^2 + \scalar{x(t) - z}{\ddot{x}(t)} \\
			&\stackrel{\eqref{LD}}{=} \norm{\dot{x}(t)}^2 + \scalar{x(t) - z}{-\sqrt{E(t)}\dot{x}(t) - \nabla f(x(t))}  \\
			&= \norm{\dot{x}(t)}^2 + \scalar{x(t) - z}{-\sqrt{E(t)}\dot{x}(t)} - \scalar{x(t) - z}{\nabla f(x(t))}  \\
			&\overset{\eqref{eq::hztdot}}{\leq } \norm{\dot{x}(t)}^2 - \sqrt{E(t)}\dot{h_z}(t) + f^\star - f(x(t)),
			\end{split}
		\end{align}
		where we used the first-order characterization of the convexity of $ f $ in the last step. So $ h_z(t) $ fulfills the following differential inequality:
		\begin{equation*}
			\ddot{h_z}(t) +\underset{=E(t)}{\underbrace{f(x(t)) - f^\star + \frac{1}{2}\norm{\dot{x}(t)}^2}} \leq \norm{\dot{x}(t)}^2 + \frac{1}{2}\norm{\dot{x}(t)}^2 - \sqrt{E(t)}\dot{h_z}(t),
		\end{equation*}
		or equivalently,
		\begin{equation}\label{eq::anchorfunctionode}
			E(t) \leq \frac{3}{2}\norm{\dot{x}(t)}^2 - \ddot{h_z}(t) - \sqrt{E(t)}\dot{h_z}(t).
		\end{equation}
		We integrate \eqref{eq::anchorfunctionode} from $ t_0 $ to $ T > t_0 $:
		\begin{align}\label{eq::anchorfunctionodeint}
		\begin{split}
			\int_{t_0}^{T}E(t)\dt &\leq
			\frac{3}{2}\int_{t_0}^{T}\norm{\dot{x}(t)}^2\dt - \int_{t_0}^{T}\ddot{h_z}(t)\dt - \int_{t_0}^{T}\sqrt{E(t)}\dot{h_z}(t)\dt\\
			&= 3\left(\sqrt{E(t_0)}- \sqrt{E(T)}\right) - \dot{h_z}(T) + \dot{h_z}(t_0) - \sqrt{E(t)}h_z(T) + \sqrt{E(0)}h_z(t_0)
			\\
			&\qquad- \frac{1}{2}\int_{t_0}^{T}\norm{\dot{x}(t)}^2h_z(t)\dt,
		\end{split}
		\end{align}
		where we performed integration by parts on the last integral and used Lemma \ref{lem::derEanddamping}. Note that by the boundedness of $ E $ and the continuity of $ f $, $ f(x) -f^\star $ and $ \norm{\dot{x}} $ are uniformly bounded on $[t_0,+\infty[$. This implies together with Assumption \ref{ass::minsol} that $ h_z $ and $ \dot{h_z} $ are uniformly bounded as well. Therefore, $ \dot{h_z}(T) $, $ \sqrt{E(T)} $ and $ \sqrt{E(T)}h_z(T) $ are uniformly bounded from above for all $ T \in [t_0,+\infty[ $. Further $ -\int_{t_0}^{T}\norm{\dot{x}(s)}^2h_z(s)\ds\leq 0$ and $ \dot{h_z}(t_0) $ and $ \sqrt{E(t_0)}h_z(t_0) $ are constants. So the right-hand side in \eqref{eq::anchorfunctionodeint} is uniformly bounded from above for all $ T \in [t_0,+\infty[ $. Therefore we deduce that
		\begin{equation}\label{eq::Eintegrable}
			\int_{t_0}^{+\infty}E(t)\dt < +\infty.
		\end{equation}
		Finally, $ E $ is non-negative and non-increasing by \eqref{eq::derE}, so it \emph{converges} to some value $ E_\infty \in [0,E(t_0)] $ as $ t \to \infty $. Further since $ E $ is continuous we can conclude by Lemma \ref{lem::convto0} that $\displaystyle \lim_{t \to +\infty}E(t) = 0. $
	\end{proof}


	\subsection{Rates of convergence for \texorpdfstring{$E$}{E}}
	The proof of our main result Theorem~\ref{thm::fastratedelta} relies on the following lemma.
	\begin{lemma}\label{lem::integrable_littleo}
		Let $ \alpha \geq 0 $: If $\int_{t_0}^{+\infty}t^\alpha E(t)\dt < + \infty$, then $\displaystyle \lim_{t \to +\infty}t^{\alpha+1}E(t) = 0$.
	\end{lemma}
	This result follows from standard calculus arguments and is proved later in Appendix \ref{app::proofsoflemmata}. \\
	Observe that \eqref{eq::Eintegrable} already provides us with a first rate by applying Lemma \ref{lem::integrable_littleo} with $ \alpha = 0 $.
	\begin{proposition}\label{prop::rate_n=1} Under Assumption \ref{ass::f} and \ref{ass::minsol}, it holds that
		  \begin{equation*}
		  	E(t) = o\left(\frac{1}{t}\right).
		  \end{equation*}
	\end{proposition}
	Since $ E $ is a sum of non-negative quantities, any convergence rate of $ E $ translates into a rate for $ f $, hence in particular Proposition \ref{prop::rate_n=1} implies $ f(x(t)) - f^\star = o\left(\frac{1}{t}\right)$.
	The proof of Theorem \ref{thm::fastratedelta} follows similar steps as that of Theorem \ref{thm::convto0}. The main idea is that by combining \eqref{eq::Eintegrable} with \eqref{eq::dersqrtE} we can improve the rate from $ o\left(\frac{1}{t}\right) $ to $ o\left(\frac{1}{t^{3/2}}\right) $ and iteratively repeat this process. This is stated in the following theorem.
	\begin{theorem}\label{thm::fastrate}
		Under Assumptions \ref{ass::f} and \ref{ass::minsol} for all $ 0 < \varepsilon < \frac{1}{2} $ and all $ n \in \mathbb{N}_{\geq 1} $ it holds that
		\begin{equation}\label{eq::fastratealpha}
			E(t) = o\left(\frac{1}{t^{\alpha_n - \alpha_{n-1}\varepsilon}}\right), \ \text{where} \ \alpha_n \eqdef 2 - \left(\frac{1}{2}\right)^{n-1}.
		\end{equation}
	\end{theorem}
	Before proving Theorem \ref{thm::fastrate}, we show that Theorem \ref{thm::fastratedelta} is a direct consequence of this result.
	\begin{proof}[Proof of Theorem \ref{thm::fastratedelta}]
		For any $ 0 < \delta < 1 $, choose $ 0 < \varepsilon < \frac{\delta}{2} $ and observe that
		\begin{equation*}
				\lim_{n \to +\infty} \alpha_n = 2 \text{, hence }
				\lim_{n \to +\infty} \left(\alpha_{n} - \alpha_{n-1}\varepsilon\right) = 2-2\varepsilon > 2-\delta,
		\end{equation*}
		which shows that there exists $ N \in \mathbb{N}_{\geq 1} $, such that $ \alpha_{N} - \alpha_{N-1}\varepsilon > 2 - \delta $ and therefore by Theorem \ref{thm::fastrate}
		\begin{equation*}
			0 = \lim_{t \to +\infty}t^{\alpha_{N}-\alpha_{N-1}\varepsilon}E(t) \geq \lim_{t \to +\infty}t^{2-\delta}E(t) \geq 0,
		\end{equation*}
		hence
		\begin{equation*}
			\lim_{t \to +\infty}t^{2-\delta}E(t) = 0.
		\end{equation*}
		Finally the case $ \delta \geq 1 $ is covered by Proposition \ref{prop::rate_n=1} and, as previously discussed, the following implication holds for all $ \delta > 0 $:
		\begin{equation*}
			E(t) = o\left(\frac{1}{t^{2-\delta}}\right) \Longrightarrow f(x(t))-f^\star = o\left(\frac{1}{t^{2-\delta}}\right). 
		\end{equation*}
	\end{proof}

	It now only remains to prove Theorem \ref{thm::fastrate}. We make use of the following lemma.
	\begin{lemma}\label{lem::littleo_integrable}
		Let $ g$ be continuous and non-negative on $ [t_0,+\infty[ $. Then for any $ \beta >1 $,
		\begin{equation*}
			\lim_{t \to +\infty} t^\beta g(t)  \text{ exists and is finite }\Longrightarrow \int_{t_0}^{+\infty}g(t)\dt < +\infty.
		\end{equation*}
	\end{lemma}
	The proof of Lemma \ref{lem::littleo_integrable} is postponed to Appendix \ref{app::proofsoflemmata}
	\begin{proof}[Proof of Theorem \ref{thm::fastrate}]
		Fix $ 0 < \varepsilon < \frac{1}{2} $. We show \eqref{eq::fastratealpha} by induction over $ n \in \mathbb{N}_{>1} $. The case $ n=1 $ holds from Proposition \ref{prop::rate_n=1}. Observe that for all $ n \geq 1 $, we have the following:
		\begin{equation}\label{eq::alphan}
			\alpha_{n} - 1 = 1 - \left(\frac{1}{2}\right)^{n-1} = \frac{1}{2}\left(2 - \left(\frac{1}{2}\right)^{n-2}\right) = \frac{\alpha_{n-1}}{2}.
		\end{equation}
		Let us now assume that there exists $ n \geq 1 $ such that:
		\begin{equation}\label{I.H}
			\tag{I.H.}
			E(t) = o\left(\frac{1}{t^{\alpha_n - \alpha_{n-1}\varepsilon}}\right) \iff \sqrt{E(t)} = o\left(\frac{1}{t^{\frac{\alpha_n}{2} - \frac{\alpha_{n-1}}{2}\varepsilon}}\right).
		\end{equation}
		To proceed by induction we need to show that \eqref{I.H} implies
		\begin{equation}\label{eq::inductiongoal}
			E(t) = o\left(\frac{1}{t^{\alpha_{n+1}-\alpha_n\varepsilon}}\right).
		\end{equation}
		To this aim, it is actually sufficient to prove that
		\begin{equation}\label{eq::EintegrableIS}
			\int_{t_0}^{+\infty}t^{\frac{\alpha_n}{2}-\alpha_n\varepsilon}E(t)\dt < + \infty.
		\end{equation}
		Indeed, according to Lemma \ref{lem::integrable_littleo} if \eqref{eq::EintegrableIS} holds, then
		\begin{equation*}
				0 = \lim_{t \to +\infty}t^{\frac{\alpha_n}{2}+1-\alpha_n\varepsilon}E(t) \overset{\eqref{eq::alphan}}{=} \lim_{t \to +\infty}t^{\alpha_{n+1}-\alpha_n\varepsilon}E(t),
		\end{equation*}
		which is equivalent to \eqref{eq::inductiongoal}.

		To show \eqref{eq::EintegrableIS} we multiply both sides of \eqref{eq::anchorfunctionode} by $ t^{\frac{\alpha_{n}}{2}-\alpha_n\varepsilon} $ and integrate from $ t_0 $ to $ T \geq t_0$:
		\begin{multline}\label{eq::hztintn}
			\int_{t_0}^{T}t^{\frac{\alpha_n}{2}-\alpha_n\varepsilon}E(t)\dt
			\leq \frac{3}{2}\int_{t_0}^{T}t^{\frac{\alpha_n}{2}-\alpha_n\varepsilon}\norm{\dot{x}(t)}^2\dt - \int_{t_0}^{T}t^{\frac{\alpha_n}{2}-\alpha_n\varepsilon}\ddot{h_z}(t)\dt \\ - \int_{t_0}^{T}t^{\frac{\alpha_n}{2}-\alpha_n\varepsilon}\sqrt{E(t)}\dot{h_z}(t)\dt.
		\end{multline}
		Now it only remains to show that the limit as $ T \to +\infty $ of each term on the right-hand side of \eqref{eq::hztintn} is uniformly bounded from above for $T\in[t_0,+\infty[$, which will imply that the left-hand side remains uniformly bounded as $ T \to +\infty $. This approach is similar to the proof of Theorem \ref{thm::convto0}, but the key equation \eqref{eq::dersqrtE} allows deducing further integrability results for $ \norm{\dot{x}(t)}^2 $ by using the induction hypothesis \eqref{I.H}, which is new and specific to the system \eqref{LD}. \\
		We first make the following observation using \eqref{eq::alphan}:
		\begin{equation*}
			\frac{t^{1+\varepsilon}}{t^{1-\frac{\alpha_n}{2} + \alpha_n\varepsilon}}\sqrt{E(t)} = t^{\frac{\alpha_n}{2} - \frac{\alpha_{n-1}}{2}\varepsilon}\sqrt{E(t)} \xrightarrow[t \to +\infty]{\eqref{I.H}} 0.
		\end{equation*}
		Therefore according to Definition \ref{def::biglittleo}, we have
		\begin{equation}\label{eq::integrability IS}
			\frac{1}{t^{1-\frac{\alpha_n}{2} + \alpha_n\varepsilon}}\sqrt{E(t)} = o\left(\frac{1}{t^{1+\varepsilon}}\right),
		\end{equation}
		and using Lemma \ref{lem::littleo_integrable} with $ g(t) = \frac{1}{t^{1-\frac{\alpha_n}{2} + \alpha_n\varepsilon}}\sqrt{E(t)}$ and $ \beta = 1+\varepsilon $,
		\begin{equation}\label{eq::integrable_arbitrary_n}
			\int_{t_0}^{+\infty} \frac{1}{t^{1-\frac{\alpha_n}{2} + \alpha_n\varepsilon}}\sqrt{E(t)}\dt<+ \infty .
		\end{equation}
		We now use \eqref{eq::integrable_arbitrary_n} to show the integrability of $ t^{\frac{\alpha_n}{2}-\alpha_n\varepsilon}\norm{\dot{x}(t)}^2 $, which shows the uniform boundedness of the first term on the right-hand side of \eqref{eq::hztintn} as $ T \to + \infty $. For any $ T \geq t_0 $:
		\begin{align*}
			&\int_{t_0}^{T} t^{\frac{\alpha_n}{2}-\alpha_n\varepsilon}\norm{\dot{x}(t)}^2\dt
			\\
			&\stackrel{\text{I.P., } \eqref{eq::dersqrtE}}{=} \left[-2t^{\frac{\alpha_n}{2}-\alpha_n\varepsilon}\sqrt{E(t)}\right]_{t_0}^{T} - \int_{t_0}^{T} \left(\frac{\alpha_n}{2}-\alpha_n\varepsilon\right)t^{\frac{\alpha_n}{2}-1-\alpha_n\varepsilon}(-2\sqrt{E(t)})\dt
			\\
			&= -2T^{\frac{\alpha_n}{2}-\alpha_n\varepsilon}\sqrt{E(T)} + 2t_0^{\frac{\alpha_n}{2}-\alpha_n\varepsilon}\sqrt{E(t_0)} +  2\left(\frac{\alpha_n}{2}-\alpha_n\varepsilon\right)\int_{t_0}^{T} \frac{1}{t^{1-\frac{\alpha_n}{2} + \alpha_n\varepsilon}}\sqrt{E(t)}\dt.
		\end{align*}
		The second term in the line above is constant, the last integral is finite as $ T \to +\infty $ by \eqref{eq::integrable_arbitrary_n} and
		\begin{equation*}
				-2\lim_{T \to +\infty}T^{\frac{\alpha_n}{2}-\alpha_n\varepsilon}\sqrt{E(T)} \leq 0.
		\end{equation*}
		Therefore the first term in \eqref{eq::hztintn} is bounded:
		\begin{equation}\label{speed integrable arbitrary n}
			\int_{t_0}^{+\infty} t^{\frac{\alpha_n}{2}-\alpha_n\varepsilon}\norm{\dot{x}(t)}^2\dt < + \infty.
		\end{equation}
		Looking at the second term in \eqref{eq::hztintn} we have:
		\begin{multline}\label{eq::lastterm}
			-\int_{t_0}^{T}t^{\frac{\alpha_n}{2}-\alpha_n\varepsilon}\ddot{h_z}(t)\dt \overset{I.P.}{=}-\left[t^{\frac{\alpha_n}{2}-\alpha_n\varepsilon}\dot{h_z}(t)\right]_{t_0}^{T} + \left(\frac{\alpha_n}{2}-\alpha_n\varepsilon\right)\int_{t_0}^{T}t^{\frac{\alpha_n}{2}- 1 -\alpha_n\varepsilon}\dot{h_z}(t)\dt \\
			\overset{I.P.}{=} -T^{\frac{\alpha_n}{2}-\alpha_n\varepsilon}\dot{h_z}(T) + t_0^{\frac{\alpha_n}{2}-\alpha_n\varepsilon}\dot{h_z}(t_0) + \left(\frac{\alpha_n}{2}-\alpha_n\varepsilon\right)\left[t^{\frac{\alpha_n}{2} -1-\alpha_n\varepsilon}h_z(t)\right]_{t_0}^{T} \\
			-\left(\frac{\alpha_n}{2}-\alpha_n\varepsilon\right)\left(\frac{\alpha_n}{2}-1-\alpha_n\varepsilon\right)\int_{t_0}^{T}t^{\frac{\alpha_n}{2}-2-\alpha_n\varepsilon}h_z(t)\dt.
		\end{multline}
		Recall that there exists a $ 0\leq M<+\infty $ such that$ \ \forall t\geq t_0  $ $ h_z(t)\in [0,M] $, since by Assumption~\ref{ass::minsol}, $ x $ is bounded. Due to $ - \ \underset{>0}{\underbrace{\left(\frac{\alpha_n}{2}-\alpha_n\varepsilon\right)}}\underset{<0}{\underbrace{\left(\frac{\alpha_n}{2}-1-\alpha_n\varepsilon\right)}} >0 $ and $ \frac{\alpha_{n}}{2} -2 -\alpha_{n}\varepsilon< -1 $, we can bound the limit of the last term in \eqref{eq::lastterm} as $ T \to +\infty $:
		\begin{multline*}
			-\left(\frac{\alpha_n}{2}-\alpha_n\varepsilon\right)\left(\frac{\alpha_n}{2}-1-\alpha_n\varepsilon\right)\int_{t_0}^{+\infty}t^{\frac{\alpha_n}{2}-2-\alpha_n\varepsilon}h_z(t)\dt \\
			\leq -M\left(\frac{\alpha_n}{2}-\alpha_n\varepsilon\right)\left(\frac{\alpha_n}{2}-1-\alpha_n\varepsilon\right)\int_{t_0}^{+\infty}t^{\frac{\alpha_n}{2}-2-\alpha_n\varepsilon}\dt < +\infty.
		\end{multline*}
		Again by boundedness of $ x $ and by the definition of $ \dot{h_z} $ we get:
		\begin{equation*}
			\begin{split}
				-\lim_{T \to +\infty}T^{\frac{\alpha_n}{2}-\alpha_n\varepsilon}\dot{h_z}(T) &\stackrel{\text{C.S.}}{\leq}  \lim_{T \to +\infty}T^{\frac{\alpha_n}{2}-\alpha_n\varepsilon}\norm{x(T)-z}\norm{\dot{x}(T)} \\
				&\leq \sqrt{M} \lim_{T \to +\infty}T^{\frac{\alpha_n}{2}-\alpha_n\varepsilon}\norm{\dot{x}(T)} = 0.
			\end{split}
		\end{equation*}
		Indeed,
		\begin{multline*}
			0\overset{\text{\eqref{I.H}}}{=}\lim_{t \to +\infty}t^{\alpha_n - \alpha_{n-1}\varepsilon}E(t) = \lim_{t \to +\infty}t^{\alpha_n - \alpha_{n-1}\varepsilon}\left(f(x(t))-f^\star + \frac{1}{2}\norm{\dot{x}(t)}^2\right) \\
			\geq \lim_{t \to +\infty}t^{\alpha_n - \alpha_{n-1}\varepsilon}\norm{\dot{x}(t)}^2 \geq 0,
		\end{multline*}
		hence:
		\begin{align*}
			\lim_{t \to +\infty}t^{\alpha_n - \alpha_{n-1}\varepsilon}\norm{\dot{x}(t)}^2 = 0
			\iff& \lim_{t \to +\infty}t^{\frac{\alpha_n}{2} - \frac{\alpha_{n-1}}{2}\varepsilon}\norm{\dot{x}(t)} = 0 \\
			\Rightarrow \lim_{t \to +\infty}t^{\frac{\alpha_n}{2} - \frac{\alpha_{n-1}}{2}\varepsilon - \varepsilon}\norm{\dot{x}(t)} = 0
			\overset{\eqref{eq::alphan}}{\iff}& \lim_{t \to +\infty}t^{\frac{\alpha_n}{2} - \alpha_{n}\varepsilon}\norm{\dot{x}(t)} = 0.
		\end{align*}
		To conclude that the limit $ T \to + \infty $ of the right-hand side of \eqref{eq::lastterm} is bounded from above, we observe that two of the three remaining terms are constant and
		\begin{equation*}
				\lim_{T \to +\infty} \left(\frac{\alpha_n}{2}-\alpha_n\varepsilon\right) T^{\frac{\alpha_n}{2} -1-\alpha_n\varepsilon}h_z(T) = 0,
		\end{equation*}
		since $ \frac{\alpha_n}{2} -1-\alpha_n\varepsilon < 0 $ and $ h_z $ is bounded. We finish the proof by bounding the third term of \eqref{eq::hztintn}. For $ T \geq t_0 $ we have:\\
		\begin{multline}\label{eq::third,partial}
			-\int_{t_0}^{T}t^{\frac{\alpha_n}{2}-\alpha_n\varepsilon}\sqrt{E(t)}\dot{h_z}(t)\dt \overset{I.P.,\eqref{eq::dersqrtE}}{=} -\left[t^{\frac{\alpha_n}{2}-\alpha_n\varepsilon}\sqrt{E(t)}h_z(t)\right]_{t_0}^{T} \\
			+ \int_{t_0}^{T}\left(\left({\frac{\alpha_n}{2}-\alpha_n\varepsilon}\right)t^{{\frac{\alpha_n}{2}-1-\alpha_n\varepsilon}}\sqrt{E(t)}- \frac{1}{2}t^{\frac{\alpha_n}{2}-\alpha_n\varepsilon}\norm{\dot{x}(t)}^2\right)h_z(t)\dt
			\\
			\hspace{-3.5cm}
			= -T^{\frac{\alpha_n}{2}-\alpha_n\varepsilon}\sqrt{E(T)}h_z(T) \ +t_0^{\frac{\alpha_n}{2}-\alpha_n\varepsilon}\sqrt{E(t_0)}h_z(t_0)
			\\+ \left({\frac{\alpha_n}{2}-\alpha_n\varepsilon}\right)\int_{t_0}^{T}t^{\frac{\alpha_n}{2}-1-\alpha_n\varepsilon}\sqrt{E(t)}h_z(t)\dt
			- \frac{1}{2}\int_{t_0}^{T}t^{\frac{\alpha_n}{2}-\alpha_n\varepsilon}\norm{\dot{x}(t)}^2h_z(t)dt.
		\end{multline}
		Only the third term on the right-hand side of \eqref{eq::third,partial} is neither constant nor non-positive. Using again the boundedness of $ h_z $, it holds that:
		\begin{equation*}
			\int_{t_0}^{+\infty}t^{\frac{\alpha_n}{2}-1-\alpha_n\varepsilon}\sqrt{E(t)}h_z(t)\dt  \leq M\int_{t_0}^{+\infty}t^{\frac{\alpha_n}{2}-1-\alpha_n\varepsilon}\sqrt{E(t)}\dt \overset{\eqref{eq::integrable_arbitrary_n}}{<} + \infty.
		\end{equation*}
		Hence the limit as $ T \to +\infty $ of the third term of \eqref{eq::hztintn} is bounded from above as well.
		Therefore all limits as $ T \to +\infty $ of the terms on the right-hand side of \eqref{eq::hztintn} are bounded from above, hence \eqref{eq::EintegrableIS} holds. By induction, we conclude that \eqref{eq::fastratealpha} holds for all $ n \in \mathbb{N}_{\geq 1} $.
	\end{proof}
	\begin{remark}\label{rem::comparisonhyperbolic}
		The induction step performed in Theorem \ref{thm::fastrate} resembles the proof of \cite[Lemma 3.10]{cabot2012hyperbolic} (then generalized in \citep{may2015long}). Yet, the result therein is not directly applicable to \eqref{LD} because it requires the knowledge of an asymptotic lower-bound on the damping, which we do not know a priori. Furthermore, \citet{cabot2012hyperbolic} showed how to iteratively repeat the application of this result but their proof relies on a strong integrability statement on $\norm{\dot{x}(t)}$ that may not hold in general. The specific nature of our system \eqref{LD} allows us to deduce this integrability statement instead of assuming it.
	\end{remark}


	\subsection{Drawbacks of the approach}

	Now that we presented the main benefits of using a Lyapunov function as damping in \eqref{LD}, we discuss some drawbacks of our approach, starting with the following remark.

	\begin{remark}\label{rem::fstar}
		Our system \eqref{LD} makes use of the optimal value $f^\star$. The latter might be unknown in general but is known in several practical cases such as over-determined regression with squared loss, or empirical risk minimization of some over-parameterized machine learning problems \citep{cybenko1989approximation}. We refer to \citep{boyd2003subgradient} for more examples and further discussion.
		All the results above do not require the knowledge of $f^\star$, only its existence. However Algorithm~\ref{algo::LD} presented in Section~\ref{sec::exp} is only practical when $f^\star$ is known.
	\end{remark}

	The second drawback of our approach is that we can  lower bound the convergence rate of our damping $ \sqrt{E(t)} $. This not surprising as the bound actually matches the sufficient condition from \cite{cabot2009long} mentioned in the introduction.
	\begin{proposition}\label{prop::notfaster}
	The convergence rate of the damping   $ \sqrt{E} $ of the dynamical system \eqref{LD} \emph{cannot} be asymptotically faster than $\frac{1}{t}$, \textit{i.e.}, there exists no $ \alpha > 1 $ such that $\sqrt{E(t)} = o\left(\frac{1}{t^\alpha}\right)$ for a non-constant bounded solution $  $ of \eqref{LD}.
	\end{proposition}

	\begin{remark}
		Proposition~\ref{prop::notfaster} expresses that $E$ can never vanish faster than $o(\frac{1}{t^2})$.
		However, this does not mean that the convergence rate of $ f(x(t))-f^\star $ can never be faster than the worst-case $o(\frac{1}{t^2})$; one can sometimes obtain faster convergence of function values, as we will observe in the numerical experiments.
	\end{remark}

	\begin{proof}[Proof of Proposition~\ref{prop::notfaster}]
	Assume that $ \sqrt{E(t)} = o\left(\frac{1}{t^\alpha}\right) $ for some $ \alpha > 1 $, then by Lemma~\ref{lem::littleo_integrable},
	\begin{equation}\label{eq::notfasterAss}
		\int_{t_0}^{\infty}\sqrt{E(t)}\dt < + \infty.
	\end{equation}
	Let us make a similar computation as in \cite[Prop. 3.5]{cabot2009long}:
	\begin{equation}\label{eq::notfasterCalc}
		\begin{split}
			\frac{\diff E(t)}{\dt} + 2\sqrt{E(t)}E(t) &\stackrel{\eqref{eq::derE}}{=} - \sqrt{E(t)}\norm{\dot{x}(t)}^2 + 2 \sqrt{E(t)}\left(f(x(t))- f^\star +\frac{1}{2} \norm{\dot{x}(t)}^2\right) \\
			&= 2 \sqrt{E(t)}(f(x(t))-f^\star) \geq 0.
		\end{split}
	\end{equation}
	Multiply \eqref{eq::notfasterCalc} by $ e^{2\int_{t_0}^{t}\sqrt{E(s)}\ds} > 0 $ to obtain
	\begin{equation}\label{eq::notfasterCalc2}
		\begin{split}
			0 \leq e^{2\int_{t_0}^{t}\sqrt{E(s)}\ds} \frac{\diff E(t)}{\dt} + 2 	\sqrt{E(t)}e^{2\int_{t_0}^{t}\sqrt{E(s)}\ds}E(t)
			= \frac{\diff }{\dt}\left[e^{2\int_{t_0}^{t}\sqrt{E(s)}\ds}E(t)\right].
		\end{split}
	\end{equation}
	Now integrate \ref{eq::notfasterCalc2} from $ t_0 $ to $ t $:
	\begin{equation*}
		\begin{split}
			0 &\leq e^{2\int_{t_0}^{t}\sqrt{E(s)}\ds}E(t)- E(t_0) \\
			\iff  E(t) &\geq e^{-2\int_{t_0}^{t}\sqrt{E(s)}\ds}E(t_0).
		\end{split}
	\end{equation*}
	This implies that
	\begin{equation*}
		\lim_{t \to \infty}E(t) \geq e^{-2\int_{t_0}^{\infty}\sqrt{E(s)}\ds}E(t_0) \overset{\eqref{eq::notfasterAss}}{>} 0,
	\end{equation*}
	which contradicts Proposition \ref{thm::convto0}, since $\displaystyle \lim_{t \to \infty} E(t) = 0.$
	Therefore, we conclude that
	\begin{align}
		\int_{t_0}^{\infty}\sqrt{E(t)}\dt = \infty,
	\end{align}
	and by Lemma \ref{lem::littleo_integrable} the limit of $ t^\alpha\sqrt{E(t)} $ for $ t \to + \infty $ does not exist or is not finite for any $ \alpha > 1 $, hence $ E(t) \neq o\left(\frac{1}{t^\alpha}\right)$ for $\alpha > 1$.
	\end{proof}


	\section{Algorithmic case and Numerical Experiments}\label{sec::exp}
	\subsection{Practical algorithms from \texorpdfstring{\eqref{LD}}{LD}}
	We first detail how we discretize \eqref{LD}.  We use an explicit discretization with fixed step-size $ \sqrt{s} > 0 $: for $ k \in \mathbb{N} $ we approximate the solution $x$ of \eqref{LD} at times $ t_k = k \sqrt{s} $ and define  $x_k\eqdef x(t_k)$.
	We use the approximations $ \dot{x}(t) \approx \frac{x_{k}-x_{k-1}}{\sqrt{s}} $ and $ \ddot{x}(t) \approx \frac{x_{k+1}-2x_k + x_{k-1}}{s} $.
	We also define the discrete version of the damping accordingly by first defining
	\begin{equation}\label{eq::discretelyapunov}
		E_k \eqdef f(x_{k})-f^\star + \frac{1}{2}\norm{\frac{x_k-x_{k-1}}{\sqrt{s}}}^2
	\end{equation}
	and our damping then reads $\gamma(x_{k},x_{k-1})\eqdef \sqrt{E_k}$.
	Using this in \eqref{LD} we then propose the following discretization scheme for all $ k \in \mathbb{N} $:
	\begin{equation}\label{eq::general discretization}
		\begin{split}
		\frac{x_{k+1}-2x_{k}+x_{k-1}}{s} + \gamma(x_{k},x_{k-1}) \frac{x_{k}-x_{k-1}}{\sqrt{s}} + \nabla f(y_k) = 0 \\
		\iff x_{k+1} = x_k + \left(1-\sqrt{s}\gamma(x_{k},x_{k-1})\right)\left[x_{k} - x_{k-1}\right] - s\nabla f(y_k),
		\end{split}
	\end{equation}
	where the gradient $\nabla f$ is evaluated at $y_k =  x_k + \left(1-\sqrt{s}\gamma(x_{k},x_{k-1})\right)\left[x_{k} - x_{k-1}\right] $, in the same fashion as it is done in NAG. One can optionally rather evaluate the gradient at $x_k$.

	We call our algorithm LYDIA, for \textbf{LY}apunov \textbf{D}amped \textbf{I}nertial \textbf{A}lgorithm, it is summarized in Algorithm~\ref{algo::LD} and publicly available in python\footnote{\url{https://github.com/camcastera/lydia}}.\\
	\begin{algorithm}[H]
		\SetKwInOut{Input}{input}
		\Input{$x_0, x_{-1} \in \mathbb{R}^n$, step-size $ s > 0 $, $ k_{\max} \in \mathbb{N} $}
		\For{$ k = 1 $ to $k_{\mathrm{max}}$}{
			$ y_k \gets  x_k + \left(1-\frac{\sqrt{E_k}}{\sqrt{E_0}}\right)\left[x_{k} - x_{k-1}\right] $ \;

			$ x_{k+1} \gets  y_{k} - s\nabla f(y_{k})$}
		\caption{LYDIA\label{algo::LD}}
	\end{algorithm}
	Note that compared to \eqref{eq::general discretization}, we actually scaled the damping term $ \sqrt{s}\sqrt{E_k} $ by $ \sqrt{s}\sqrt{E_0} $ in Algorithm~\ref{algo::LD}.
	This scaling is optional but improves numerical stability\footnote{An alternative to scaling is choosing $s$ so that $\sqrt{s}\sqrt{E_0}\leq 1$.} as it ensures that the coefficient in front of $ \left[x_{k} - x_{k-1}\right] $ remains non-negative. This fact is a consequence of the following convergence theorem.
\begin{theorem}\label{thm::LYDIAconv}
	Let $ f $ be a continuously differentiable convex function whose gradient is Lipschitz continuous with constant $ L > 0 $. Let $ \left(x_k,y_k\right)_{k \in \mathbb{N}} $ be the sequence generated by Algorithm \ref{algo::LD}, where $ s \leq  \frac{1}{L} $. Then, for all $ k \in \mathbb{N} $,
	\renewcommand{\theenumi}{{(\roman{enumi})}}%
	\begin{enumerate}
		\itemsep0em
		\item $(E_k)_{k\in\mathbb{N}}$ is non-increasing, \label{item::Ek}
		\item $f(x_k)$ converges and $\lim_{k \to +\infty} \Vert x_k-x_{k-1}\Vert = 0$, \label{item::f}
		\item $\lim_{k \to +\infty} \Vert \nabla f(x_k)\Vert = 0$. \label{item::grad}
	\end{enumerate}
	If furthermore $\mathcal{H}$ has finite dimension and $(x_k)_{k \in \mathbb{N}}$ admits a bounded sub-sequence, then $\lim_{k \to +\infty} f(x_k) = f^\star$ and $\lim_{k \to +\infty} E_k = 0$.
\end{theorem}

	\begin{proof}
		For any $ x,y \in \mathcal{H} $ and $ 0<s \leq 1/L$, the following refined version of the descent lemma holds  \citep[Remark 5.2.]{attouch2022ravine}:
		\begin{equation}\label{eq::refineddescentlemma}
			f(y-s\nabla f(y)) \leq f(x) + \scalar{\nabla f (y)}{y-x} -\frac{s}{2} \norm{\nabla f(y)}^2 - \frac{s}{2}\norm{\nabla f(x) - \nabla f(y)}^2.
		\end{equation}
		Take $ y = y_k $ and $ x = x_k $ in \eqref{eq::refineddescentlemma} and subtract $ f^\star $ on both sides:
		\begin{multline}\label{eq::descentlemmapluggedin}
			f(x_{k+1}) - f^\star \leq f(x_k) - f^\star + \scalar{\nabla f (y_k)}{y_k-x_k} \\- \frac{s}{2} \norm{\nabla f(y_k)}^2 - \frac{s}{2}\norm{\nabla f(x_k) - \nabla f(y_k)}^2.
		\end{multline}
		We use the notation $\beta_k \stackrel{\mathrm{def}}{=} \sqrt{\frac{E_k}{E_0}}$. Observe that we have the following identities:
		\begin{align}
			\norm{x_{k+1}- x_{k}}^2 &= \norm{s \nabla f(y_k) - (y_k - x_k)}^2 \nonumber\\
			&= s^2\norm{\nabla f(y_k)}^2 -2s\scalar{\nabla f(y_k)}{y_k-x_k} + \norm{y_k-x_k}^2, \quad \text{and,} \label{eq::basic1} \\
			\norm{y_k-x_k}^2 &\stackrel{\eqref{algo::LD}}{=} \norm{\left(1-\beta_k\right)\left(x_k-x_{k-1}\right)}^2. \label{eq::basic2}
		\end{align}
		Therefore,
		\begin{align*}
			\langle\nabla f(y_k),y_k - &x_k\rangle = \frac{s}{2}\norm{\nabla f(y_k)}^2 - \frac{1}{2s} \norm{x_{k+1}- x_{k}}^2 + \frac{1}{2s}\norm{y_k-x_k}^2 \\
			=& \frac{s}{2}\norm{\nabla f(y_k)}^2 - \frac{1}{2s} \norm{x_{k+1}- x_{k}}^2  + \frac{1}{2s}\norm{\left(1-\beta_k\right)\left(x_k-x_{k-1}\right)}^2.
		\end{align*}
		Substitute this in \eqref{eq::descentlemmapluggedin} to get
		\begin{multline}\label{eq::generalk}
			f(x_{k+1}) - f^\star + \frac{1}{2s}\norm{x_{k+1}- x_{k}}^2 \leq f(x_k) - f^\star + \frac{1}{2s}\norm{\left(1-\beta_k\right)\left(x_k-x_{k-1}\right)}^2 \\ - \frac{s}{2}\norm{\nabla f(x_k) - \nabla f(y_k)}^2.
		\end{multline}
		Expanding $\left(1-\beta_k\right)^2= 1 - 2\beta_k + \beta_k^2$, we obtain
		\begin{multline}\label{eq::generalk2}
			f(x_{k+1}) - f^\star + \frac{1}{2s}\norm{x_{k+1}- x_{k}}^2 \leq f(x_k) - f^\star + \frac{1}{2s}\norm{x_k-x_{k-1}}^2 \\
			- \frac{1}{s}\beta_k\left(1-\frac{\beta_k}{2}\right)\norm{x_k-x_{k-1}}^2 - \frac{s}{2}\norm{\nabla f(x_k) - \nabla f(y_k)}^2.
		\end{multline}
		Using the definition of $E_k$ we finally have
		\begin{equation}\label{eq::Ekdecrease}
			E_{k+1} \leq E_k - \frac{1}{s}\beta_k\left(1-\frac{\beta_k}{2}\right)\Vert x_k-x_{k-1}\Vert^2 - \frac{s}{2}\norm{\nabla f(x_k) - \nabla f (y_k)}^2.
		\end{equation}

		\renewcommand{\theenumi}{{($\roman{enumi}$)}}%
		\begin{enumerate}
		\item Observe in the case $k=0$ that \eqref{eq::Ekdecrease} implies $E_1\leq E_0$, and more generally,
		whenever $E_k\leq E_0$ then $\beta_k=\sqrt{E_k/E_0}\leq 1$ so $E_{k+1}\leq E_k\leq E_0$. Therefore by induction, for all $k\in\mathbb{N}$, $E_{k+1}\leq E_k$ which proves $\ref{item::Ek}$, and $E_k$ converges.
		\item Summing \eqref{eq::Ekdecrease} from $0$ to $k-1$ yields
		\begin{equation}\label{eq::SumEk}
			\sum_{i=0}^{k-1} \frac{\beta_i}{s}\left(1-\frac{\beta_i}{2}\right)\Vert x_i-x_{i-1}\Vert^2 + \frac{s}{2}\norm{\nabla f(x_i) - \nabla f (y_i)}^2 \leq E_0-E_k.
		\end{equation}
		Since $E_0-E_k\leq E_0$, we deduce that $\sum_{i=0}^{+\infty}\sqrt{E_i}\Vert x_i-x_{i-1}\Vert^2<+\infty$ so $\lim_{k \to +\infty} \sqrt{E_k}\Vert x_k-x_{k-1}\Vert^2=0$. Since $E_k$ converges, either $\lim_{k \to +\infty} \Vert  x_k-x_{k-1}\Vert=0$ or $\lim_{k \to +\infty} E_k=0$, but the latter implies the former, so in every case $\lim_{k \to +\infty} \Vert  x_k-x_{k-1}\Vert=0$, so $\ref{item::f}$ holds.

		\item From \eqref{eq::SumEk}, we deduce $\lim_{k \to +\infty} \norm{\nabla f(x_k) - \nabla f (y_k)} = 0 $ similarly as above.
		Also remark, from the definition of Algorithm~\ref{algo::LD}, that $s\nabla f(y_k) = -(x_{k+1}-x_{k}) + (1-\beta_k)(x_k-x_{k-1})$, so the triangle inequality and $\ref{item::f}$ imply that $\lim_{k \to +\infty} \Vert  \nabla f(y_k)\Vert=0$.
		The two limits above yield the result.
		\end{enumerate}
		Finally, to prove the last claim, assume now that $(x_k)_{k\in\mathbb{N}}$ admits a bounded sub-sequence and that $\mathcal{H}$ is finite dimensional. Then by Bolzano-Weirstra\ss\ theorem $(x_k)_{k\in\mathbb{N}}$ has an accumulation point $\bar{x}\in\mathcal{H}$. By continuity of $\nabla f$ and $\ref{item::grad}$, $\nabla f(\bar{x})=0$ and the convexity of $f$ implies $f(\bar{x})=f^\star$. By $\ref{item::f}$, $(f(x_k))_{k\in\mathbb{N}}$ converges, so the limit is $f^\star$.
	\end{proof}

	\begin{remark}
		The results stated in Theorem~\ref{thm::LYDIAconv} are in analogy with those for \eqref{LD},
		 with the exception of the rate of convergence which will be addressed in a future work. The proof of Theorem~\ref{thm::LYDIAconv} shows how the study of the discrete algorithm is analogous to the ODE case, yet significantly more technical.
	\end{remark}

	\subsection{Numerical experiments}
		\begin{figure}[ht]
		\centering
		\includegraphics[width= 0.4\textwidth]{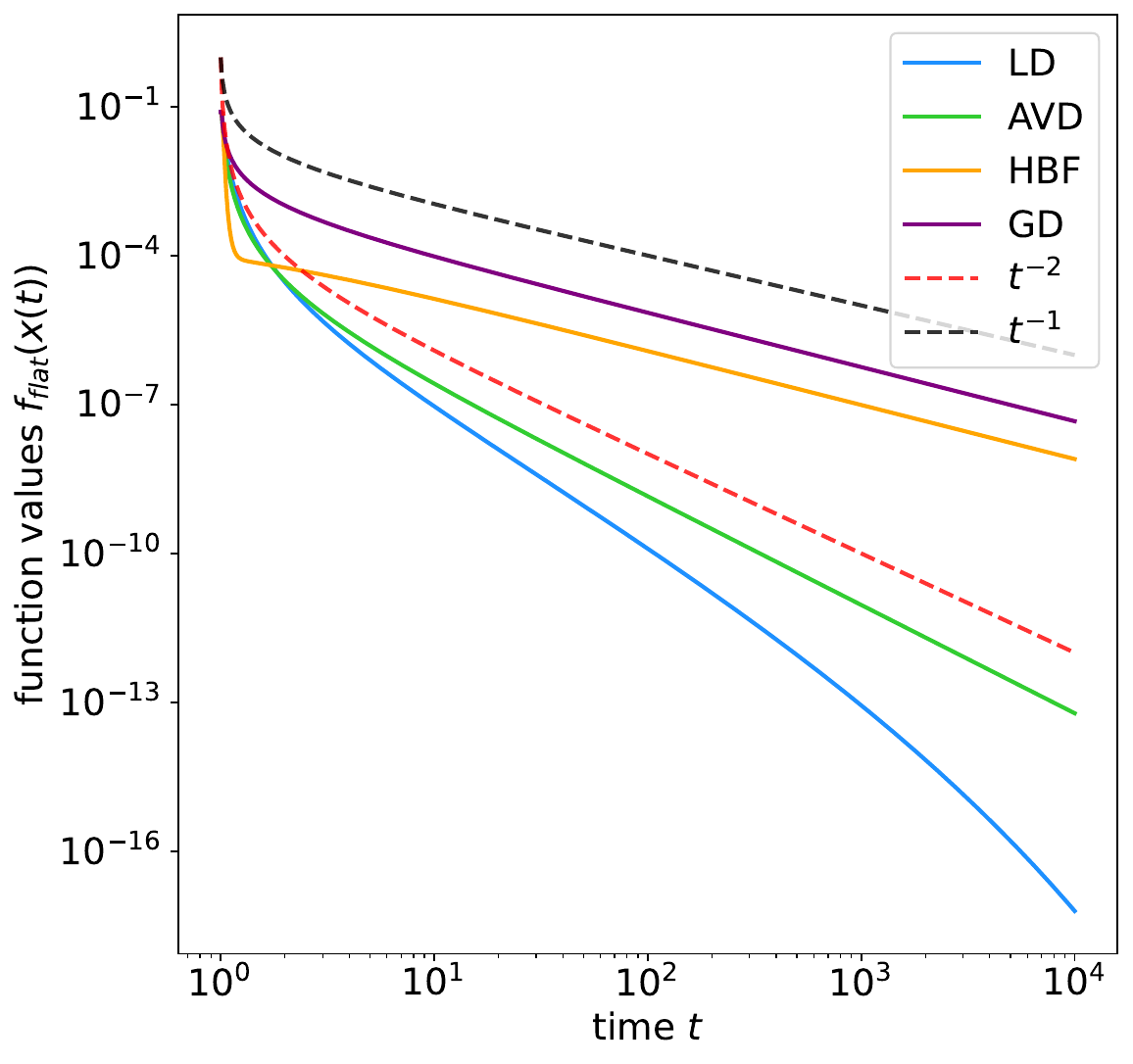}
		\includegraphics[width= 0.4\textwidth]{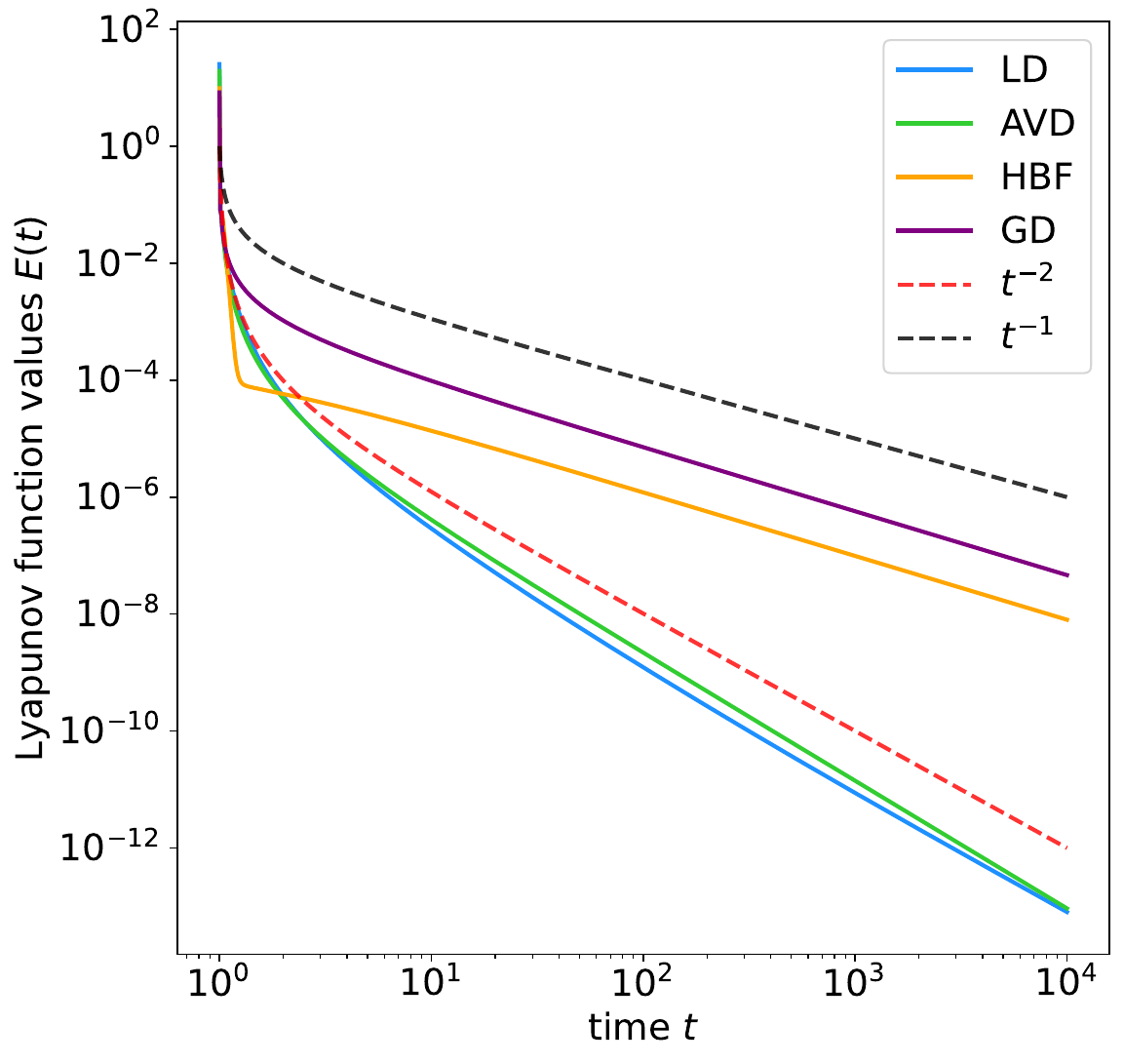}
		\caption{Rate of convergence of LD, AVD, HBF and GD on $ f_\text{flat} $. Left: Evolution of the function values. Right: Evolution of the Lyapunov function. \label{fig::rate_flatpoly}}
	\end{figure}
	\begin{figure}[ht]
		\centering
		\includegraphics[width= 0.4\textwidth]{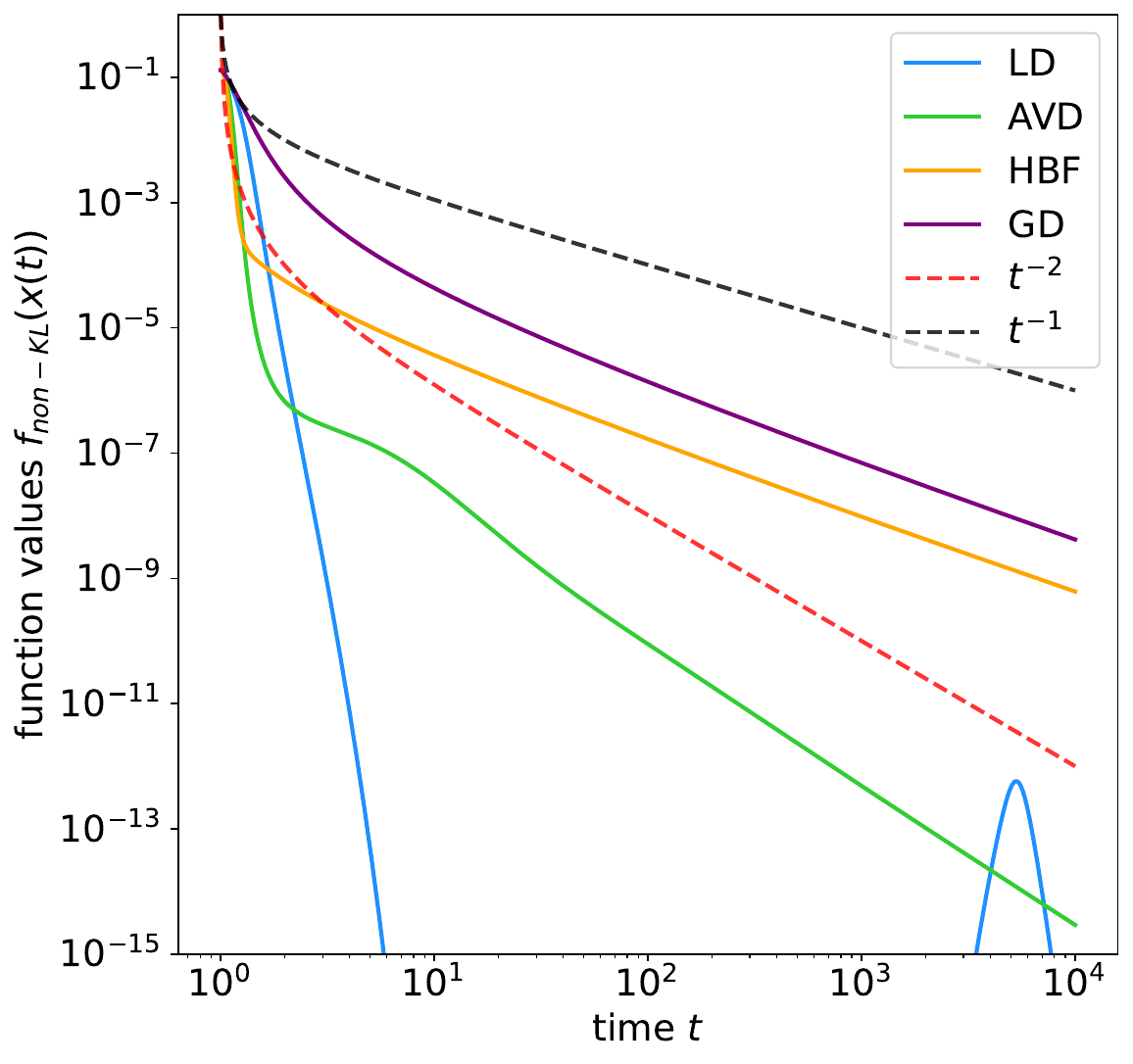}
		\includegraphics[width= 0.4\textwidth]{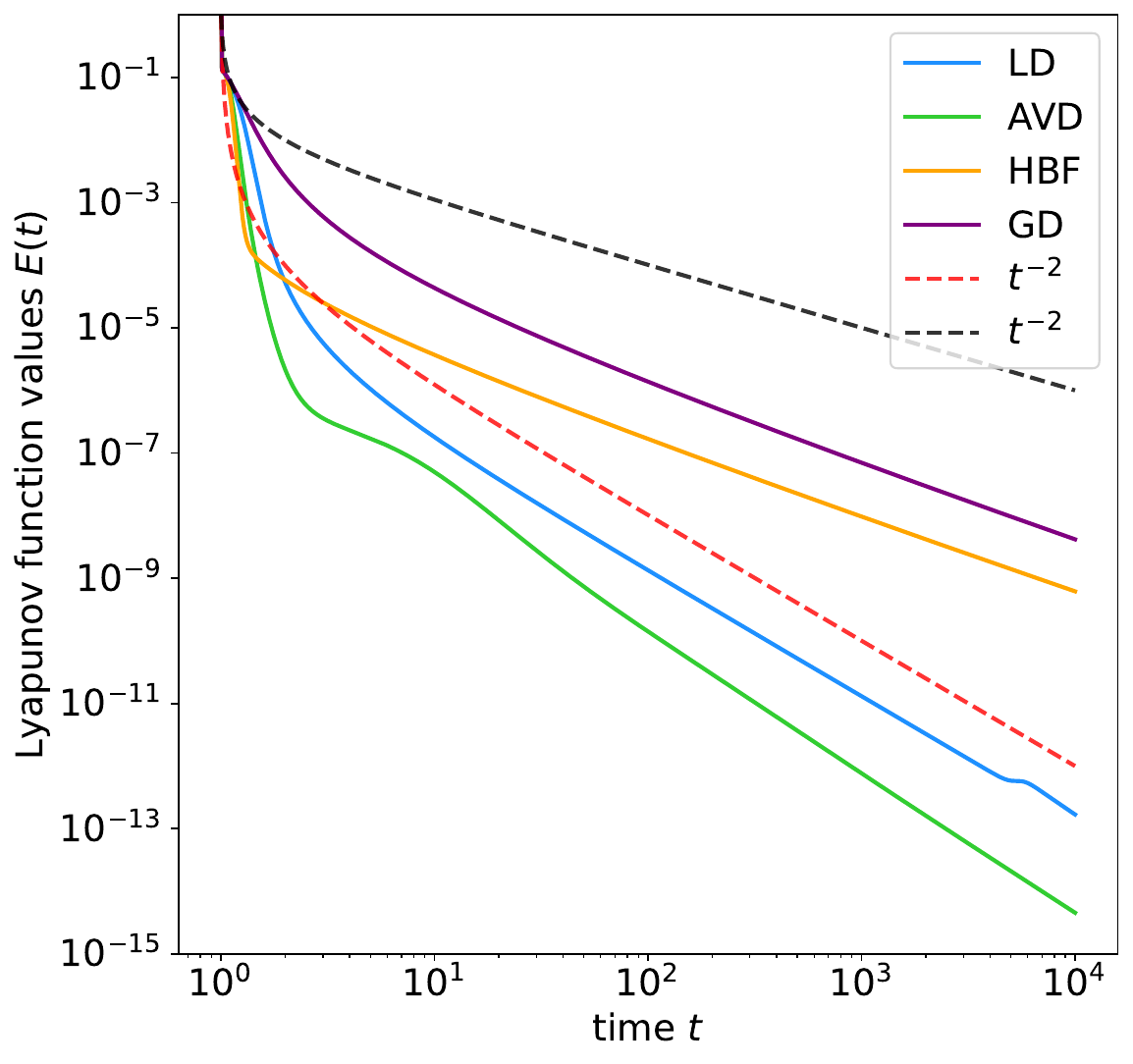}
		\caption{Rate of convergence of LD, AVD, HBF and GD on $ f_\text{non-KL} $. Left: Evolution of the function values. Right: Evolution of the Lyapunov function. \label{fig::rate_nonkl}}
	\end{figure}
	\begin{figure}[ht]
		\centering
		\includegraphics[width= 0.4\textwidth]{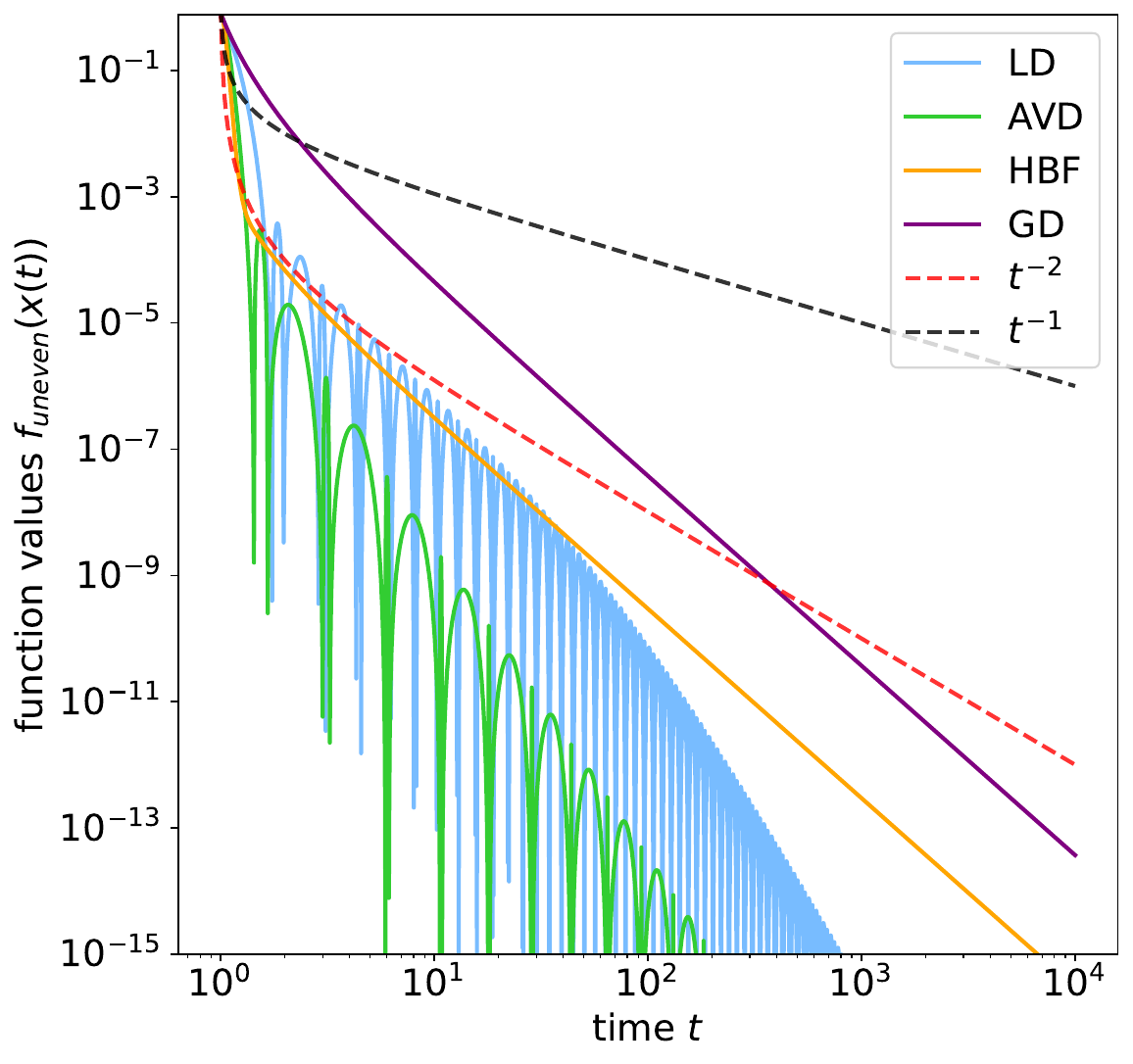}
		\includegraphics[width= 0.4\textwidth]{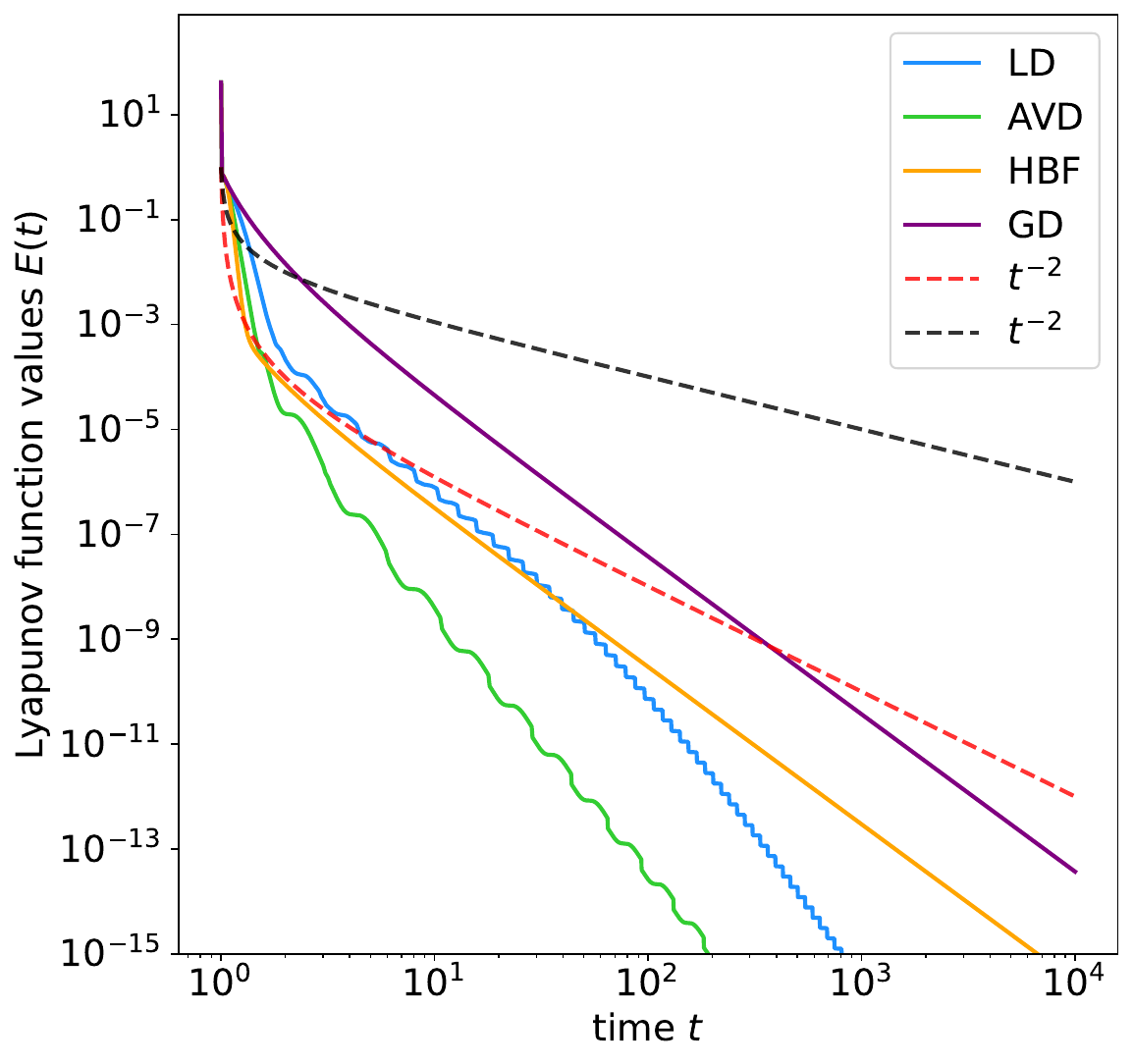}
		\caption{Rate of convergence of LD, AVD, HBF and GD on $ f_\text{uneven} $. Left: Evolution of the function values. Right: Evolution of the Lyapunov function. \label{fig::rate_uneven}}
	\end{figure}
	\begin{figure}[ht]
		\centering
		\includegraphics[width= 0.4\textwidth]{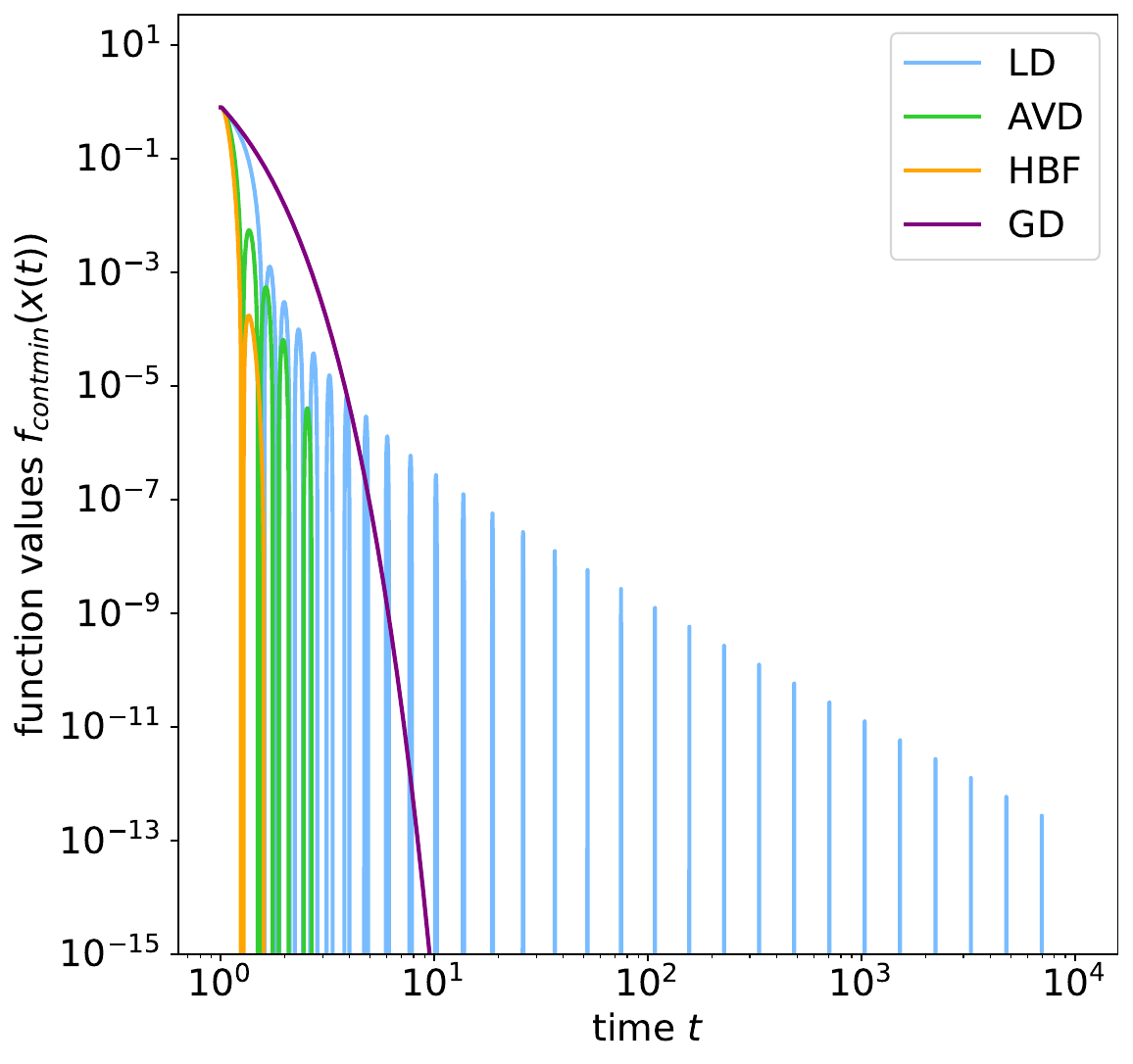}
		\includegraphics[width= 0.4\textwidth]{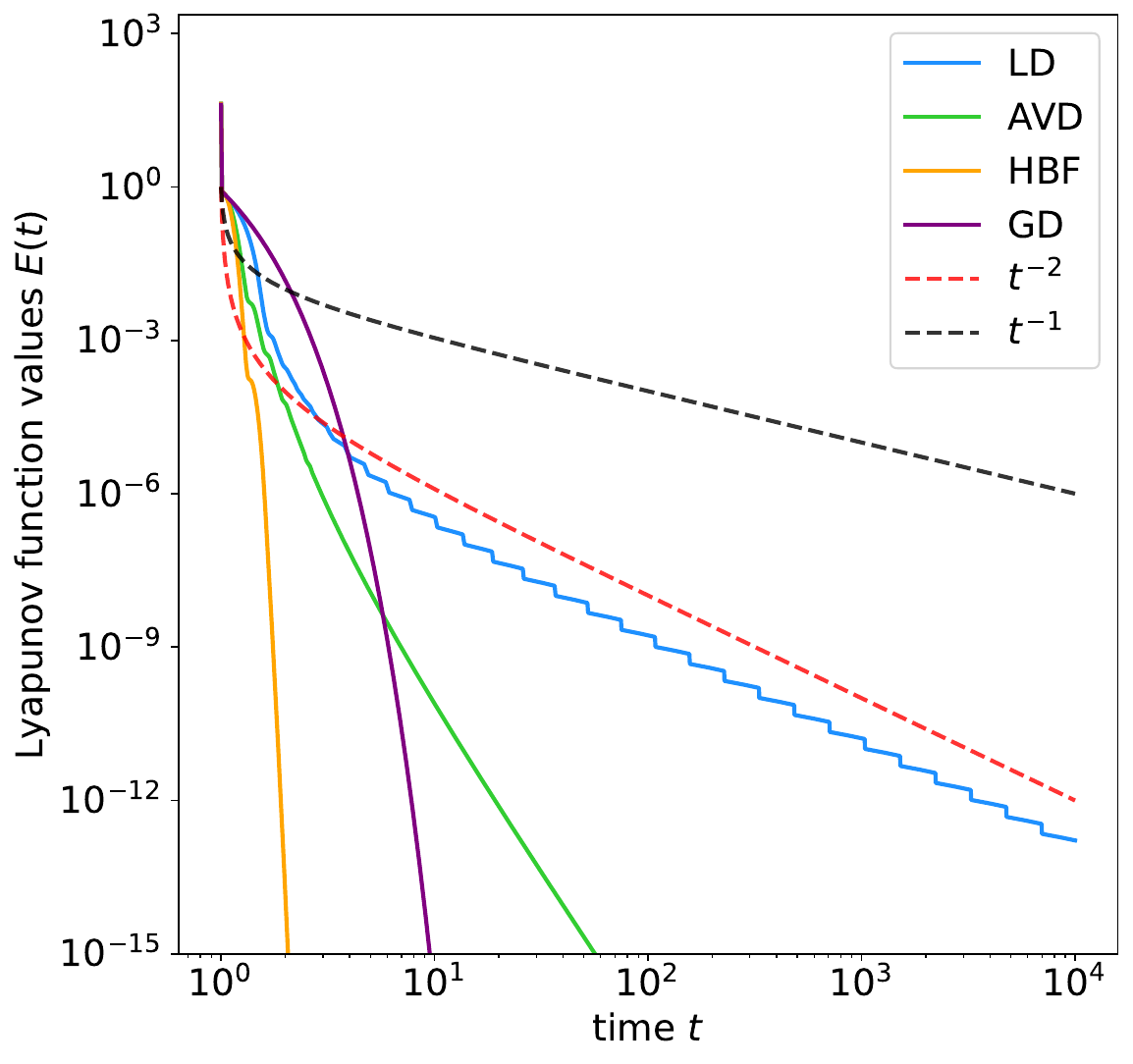}
		\caption{Rate of convergence of LD, AVD, HBF and GD on $ f_\text{contmin} $. Left: Evolution of the function values. Right: Evolution of the Lyapunov function.
			\label{fig::rate_cont_min}}
	\end{figure}
	\paragraph{Methodology}
	We evaluate the performance of our system and empirically corroborate Theorem~\ref{thm::fastratedelta}. We consider the following 1-dimensional test functions, defined for all $x\in\mathbb{R}$ by:
	\begin{align*}
		f_\text{flat}(x) &\eqdef x^{24},
		&f_\text{non-KL}(x) &\eqdef \begin{cases}
			0, & \text{if } x = 0; \\
			\exp\left({-\frac{1}{x^2}}\right), & \text{else},
		\end{cases} \\
		f_\text{uneven}(x) &\eqdef \begin{cases}
			x^3, & \text{if } x > 0; \\
			x^2, & \text{if } x \leq 0,
		\end{cases}
		&f_\text{contmin}(x) &\eqdef \begin{cases}
			(x-\varepsilon)^2, & \text{if }x>\varepsilon; \\
			(x+\varepsilon)^2, & \text{if }x<\varepsilon; \\
			0, &\text{else}.
		\end{cases}
	\end{align*}
	Since \eqref{LD} and \eqref{eq::AVD} are second-order ODEs that do not have closed-form solutions in general, we approximate them by using the LYDIA algorithm and NAG respectively, both  with small step-sizes. We take $a=3.1$ in NAG. We also consider Gradient Descent (GD) and the HBF algorithms mentioned in the introduction. For each test function\footnote{Some of these functions do not have a globally Lipschitz continuous gradient which could cause numerical instabilities. We overcome this by using small-enough step-sizes to ensure boundedness of the iterates of all the algorithms considered.}, we present the evolution of the optimality gap $f(x(t))-f^\star$ as a function of the time $t$. Since this value is not monotone and may oscillate heavily, we also display the evolution of the (discretized) Lyapunov function $E$.

	\paragraph{Results}
	The function $f_\text{flat}$ is a convex polynomial that is very flat around its minimizer. As discussed in \cite[Example 2.12]{attouch2015fast} such polynomials of large degree allow emphasizing the worst-case bounds on the rates of convergence. Indeed, we observe on the right of Figure~\ref{fig::rate_flatpoly} that the rate for GD and HBF is asymptotically close to $\mathcal{O}\left(\frac{1}{t}\right)$ while that of \ref{LD} and \eqref{eq::AVD} gets close to $\mathcal{O}\left(\frac{1}{t^2}\right)$ which empirically validates our main result Theorem~\ref{thm::fastratedelta}.

	Some convex functions are not strongly convex but possess the KL property which allows deriving faster rates of convergence (as done in the closed-loop setting by \citep{attouch2022fast}). Since Theorem~\ref{thm::fastratedelta} does not assume such a property, we evaluate the algorithms on a function that is known for not having the KL property on Figure~\ref{fig::rate_nonkl}. The results again match the theoretical ones.
	Similarly, properties such as $ f $ being even (\textit{i.e.}, $ f(-x) = f(x),\ \forall x$) or $ \argmin_{\mathcal{H}} f $ being a singleton can allow one to derive faster convergence results than for general convex functions. Our last-two test functions in Figures~\ref{fig::rate_uneven} and~\ref{fig::rate_cont_min} illustrate that the rate of \eqref{LD} also holds without such properties.

	Finally, while Figures~\ref{fig::rate_flatpoly} and~\ref{fig::rate_nonkl} evidence cases where the solution of  \eqref{LD} yields very fast minimization of function values, Figure~\ref{fig::rate_uneven} shows a case where \eqref{LD} is faster than GD and HBF but not as fast as \eqref{eq::AVD} and Figure~\ref{fig::rate_cont_min} shows a case where \eqref{LD} is behind all other methods. Nonetheless, in all cases observe that $E$ is non-increasing and that it exhibits the rate predicted by Theorem~\ref{thm::fastratedelta}.


	\section{Conclusions and Perspectives}\label{sec::conclusion}

	We introduced a new system \eqref{LD} which, as initially intended, is independent of the choice of the initial time $t_0$, unlike \eqref{eq::AVD}, and which does require the hyper-parameter $a>3$ compared to the latter, but however assumes knowledge of the optimal value $f^\star$.
	We showed that it is possible to design a closed-loop damping achieving near-optimal rates of convergence for convex optimization. The key ingredient is the coupling between the damping and the speed of convergence of the system. This yielded in particular the identity \eqref{eq::derE}, specific to \eqref{LD} and essential to adapt the proofs from the open-loop to the closed-loop setting. This coupling provides a new understanding on how to choose the damping term in second-order ODEs for optimization which may prove to be useful beyond the specific cases of \eqref{LD} and \eqref{eq::AVD}.
	While our work focuses on the ODE setting, we provided numerical experiments corroborating our theoretical results as well as a new practical first-order algorithm LYDIA, derived from \eqref{LD}, for which we also stated convergence results.

	As for future work, it remains open to know whether we can improve the rate from ``arbitrarily close to optimal'' to optimal (\textit{i.e.}, can we take $ \delta = 0 $ in Theorem~\ref{thm::fastratedelta}?).
	We also suspect that one can drop Assumption \ref{ass::minsol}-\ref{bounded} as it is not required for \eqref{eq::AVD} \citep{may2017asymptotic, jendoubi2015asymptotics}, but we could not yet do it in the closed-loop setting.
	On the algorithmic side, our experiments suggest that the rate of \eqref{LD} is transferred to the LYDIA algorithm, which remains to be shown. One might also consider other discretization schemes.
	Finally, it is worth investigating whether one can find systems and algorithms with the same properties that do not require the optimal value $f^\star$.
		
	\section*{Acknowledgment}
	C. Castera and P. Ochs are supported by the ANR-DFG joint project TRINOM-DS under number DFG-OC150/5-1.

	\FloatBarrier
	\vspace{0.5cm}
	\appendix
	\begin{flushleft}
		{\LARGE \textbf{Appendices}}
	\end{flushleft}


	\section{Proof of Existence and Uniqueness}\label{app::existence}
	This section is devoted to proving Theorem \ref{thm::existence}. We recall some result from the theory of ODEs.
	\begin{theorem}\cite[Theorem 4.2.6]{seifert2022evolutionary} (Classical Version of Picard--Lindelöf)\label{thm::piclind} \\
		Let $\mathcal{H}$ be a Hilbert space, let $\Omega \subseteq \mathbb{R} \times \mathcal{H}$ be open. Let $G: \Omega \rightarrow \mathcal{H}$ be continuous and let $\left(t_0, x_0\right) \in \Omega$. Assume that there exists $L \geq0$ such that for all $(t, x),(t, y) \in \Omega$ we have
		$$
		\norm{G(t, x)-G(t, y) } \leq L\norm{x-y} .
		$$
		Then, there exists $\delta>0$ such that the initial value problem
		$$
		\left\{\begin{array}{l}
			u^{\prime}(t)=G(t, u(t)) \quad\left(t \in\left(t_0, t_0+\delta\right)\right), \\
			u\left(t_0\right)=x_0,
		\end{array}\right.
		$$
		admits a unique continuously differentiable solution $u:\left[t_0, t_0+\delta\right[ \rightarrow \mathcal{H}$, which satisfies $(t, u(t)) \in \Omega$ for all $t \in\left[t_0, t_0+\delta\right[$.
	\end{theorem}
	\begin{definition}\cite[Definition 3.20]{philipnotes}\label{def::maxtime}
		Let $ \phi: I \to \mathbb{R}^n $, $ n \in \mathbb{N} $, be a solution to an ODE defined on some open interval $ I \subset \mathbb{R} $.
		\begin{enumerate}[label=(\alph*)]
			\item\label{it::cont} We say $ \phi $ admits an extension or continuation to the right (resp.\ to the left) if, and only if, there exists a solution $ \psi: J \to \mathbb{R}^n $ to the same ODE, defined on some open interval $ J \supset I $ such that $ \psi|_I = \phi $ and
			\begin{equation*}
				\sup J > \sup I \quad (\text{resp.\ } \inf J < \inf I)
			\end{equation*}
			An extension or continuation of $ \phi $ is a function that is an extension to the right or an extension to the left (or both).
			\item The solution $ \phi $ is called a \emph{maximal} solution if, and only if, it does not admit any extension in the sense defined in \ref{it::cont}.
		\end{enumerate} 
	\end{definition}
	\begin{theorem}\cite[Theorem 3.22]{philipnotes}\label{thm::maxtime}
		Every solution $ \phi_0: I_0 \to \mathbb{R}^n $ to some ODE, defined on an open interval $ I_0 \subset \mathbb{R} $, can be extended to a maximal solution of this ODE.
	\end{theorem}
	\begin{remark}
		Definition \ref{def::maxtime} is easily generalized to Hilbert spaces and Theorem~\ref{thm::maxtime} is valid for an ODE in an infinite dimensional space as well, since the proof relies on Zorn's Lemma.
	\end{remark}
	\begin{proof}[Proof of Theorem \ref{thm::existence}]
		We first reduce the system \eqref{LD} to an equivalent  first-order system:
		\begin{equation}\label{eq::firstorder}
			\frac{\diff }{\dt}\begin{pmatrix}
				x(t) \\
				\dot{x}(t)
			\end{pmatrix}
			= \begin{pmatrix}
				\dot{x}(t) \\
				-\sqrt{f(x(t))-f^\star+\frac{1}{2}\norm{\dot{x}(t)}^2}\dot{x}(t) - \nabla f(x(t))
			\end{pmatrix}
			\eqdef G(t,(x(t),\dot{x}(t))).
		\end{equation}
		We first show the claim in the case $ \argmin_{\mathcal{H}} f \neq \emptyset $ and later for the case $ \argmin_{\mathcal{H}} f = \emptyset $.

		First, if $\left(x_0,\dot{x}_0\right) \in \argmin_{\mathcal{H}} f \times \left\{0\right\}$, then $ (x(t),\dot{x}(t))  \equiv (x_0,0)$ is the unique global solution, since $ \ddot{x}(t)=0 $ for all $t\geq t_0$.
		Now assume that $\left(x_0,\dot{x}_0\right) \not\in \argmin_{\mathcal{H}} f \times \left\{0\right\}$ and let us check that for all $t\geq t_0$, the function
		$$ (u,v) \mapsto G(t,(u,v)) =  \begin{pmatrix}
				v \\
				-\sqrt{f(u)-f^\star+\frac{1}{2}\norm{v}^2}v(t) - \nabla f(u)
			\end{pmatrix}$$ is Lipschitz continuous on the open set
		\begin{equation*}
			\mathcal{A} \eqdef (\mathcal{H} \times \mathcal{H}) \setminus (\argmin_{\mathcal{H}} f \times \left\{0\right\}).
		\end{equation*}
		The first coordinate of $ G $ is clearly locally Lipschitz continuous w.r.t.\ $ (u,v) $ and therefore it suffices to check that the second coordinate of $ G $ is also locally Lipschitz continuous w.r.t.\ $ (u,v) $ on $ \mathcal{A} $.\\
		The square root function is locally Lipschitz continuous everywhere except at 0 and $ f(u)-f^\star+\frac{1}{2}\norm{v}^2\neq 0$ for all $(u,v)\in \mathcal{A} $. Therefore $(u,v)\in \mathcal{A} \mapsto \sqrt{f(u)-f^\star+\frac{1}{2}\norm{v}^2}v $ is locally Lipschitz continuous as a product of locally Lipschitz-continuous functions. Finally $ \nabla f $ is locally Lipschitz continuous by Assumption~\ref{ass::f}-(i) , hence $G$ is locally Lipschitz continuous on $\mathcal{A}$.
		We can now apply the Picard--Lindelöf Theorem \ref{thm::piclind} to any open and bounded neighborhood of the initial conditions $ (t_0,(x_0,\dot{x}_0)) $ to find some $ \delta > 0 $ and a unique solution $(x,\dot{x})$ to \eqref{eq::firstorder} on $ [t_0,t_0+\delta[ $.
		
		It now remains to show that the solution exists on $[t_0,+\infty[$. Let $ T \in\, ]t_0,+\infty] $ be the maximal time of existence (as stated in Theorem \ref{thm::maxtime}) of the solution $(x,\dot{x})$ of \eqref{eq::firstorder}. Assume that $ T < +\infty $, then the limits of $ x $ and $ \dot{x}(t) $ exist as $ t \to T $. Indeed, first note that the uniform boundedness of $ E $ implies that $ \dot{x} $ is bounded on $ [t_0,T[ $ and by \eqref{LD} $ \ddot{x} $ is bounded on $ [t_0,T[ $ as well. Now consider a Cauchy sequence $ \left(t_n\right)_{n\in \mathbb{N}} $ converging to $ T $ and let $ m,n \in \mathbb{N} $. Then by the mean value Theorem we have:
		\begin{equation*}
			\begin{split}
				\norm{x(t_n)-x(t_m)} &\leq \sup_{s \in ]t_0,T[}\norm{\dot{x}(s)}\left|t_n-t_m\right|, \text{ and} \\
				\norm{\dot{x}(t_n)-\dot{x}(t_m)} &\leq \sup_{s \in ]t_0,T[}\norm{\ddot{x}(s)}\left|t_n-t_m\right|.
			\end{split}
		\end{equation*}
		Therefore $ \left(x(t_n)\right)_{n\in \mathbb{N}} $ and $ \left(\dot{x}(t_n)\right)_{n\in \mathbb{N}} $ are Cauchy sequences which implies that  $ \displaystyle \lim_{t \to T}x(t) \text{ and } \lim_{t \to T}\dot{x}(t)$ exist.

		Then, if $ (x(T),\dot{x}(T)) \in \mathcal{A}$ then we can extend the solution by applying Picard--Lindelöf to an open and bounded neighborhood of $ \left((T,\left(x(T), \dot{x}(T)\right)\right) $ which contradicts the fact that $ T $ was the maximal time of existence.
		If $ (x(T),\dot{x}(T)) \in \partial \mathcal{A} \subset \argmin_{\mathcal{H}} f \times \left\{0\right\}$, where $\partial \mathcal{A}$ denotes the boundary of $\mathcal{A}$, then $ \ddot{x}(T) = 0 $ and we can extend the solution on $ [T,T+\delta[ $, for some $ \delta > 0 $, by taking $ (x(t),\dot{x}(t)) = (x(T),0) $. This is again a contradiction, therefore we necessarily have $ T = +\infty $.

		Finally, in the case where $ \argmin_{\mathcal{H}} f = \emptyset $, we see similarly to before that $ G(t,(u,v)) $ is locally Lipschitz continuous w.r.t.\ $ (u,v) $ everywhere on $ \mathcal{H} \times \mathcal{H}$. So the Picard--Lindelöf Theorem implies the existence of a unique local solution for any initial value, which we can again extend to a global solution on $ [t_0,+\infty[ $.
	\end{proof}
	\section{Missing Proofs of Lemmas}\label{app::proofsoflemmata}
	\begin{proof}[Proof of Lemma \ref{lem::convto0}]
		Assume that there exists $ c > 0 $ such that $\displaystyle \lim_{t\to+\infty}g(t) = c$, which implies the existence of $ t_1 \geq t_0 $ such that $ \forall t \geq t_1 $, $ g(t) \geq \frac{c}{2}$. Then
		$$ \int_{t_0}^{+\infty}g(t)\dt \geq \int_{t_1}^{+\infty}g(t)\dt \geq \int_{t_1}^{+\infty} \frac{c}{2}\dt = +\infty, $$ which contradicts the assumption that $\int_{t_0}^{+\infty}g(t)\dt < +\infty$.
	\end{proof}

	\begin{proof}[Proof of Lemma \ref{lem::integrable_littleo}] Denote the positive and negative parts of a function $ g $ by $$\left[g(\cdot)\right]_+ \eqdef \max(g(\cdot),0) \text{ and } \left[g(\cdot)\right]_- \eqdef \max(-g(\cdot),0) \geq 0,$$
		so that $\forall x\in\mathcal{H}$,
		 $$ g(x) = \left[g(x)\right]_+ - \left[g(x)\right]_-.$$
		Then note that
		\begin{equation}\label{eq::integrable_little0_1}
			\left[\left(t^{\alpha +1}E(t)\right)'\right]_+ \overset{\eqref{eq::derE}}{=}\left[(\alpha +1)t^\alpha E(t) - t^{\alpha+1}\sqrt{E(t)}\norm{\dot{x}(t)}^2\right]_+ = (\alpha+1)t^\alpha E(t),
		\end{equation}
		and that,
		\begin{equation}\label{eq::integrable_littleo_2}
			\begin{split}
				\int_{t_0}^{+\infty}\left(t^{\alpha+1}E(t)\right)'\dt &= \int_{t_0}^{+\infty}\left[\left(t^{\alpha +1}E(t)\right)'\right]_+ - \left[\left(t^{\alpha +1}E(t)\right)'\right]_-\dt \\
				&\leq \int_{t_0}^{+\infty}\left[\left(t^{\alpha +1}E(t)\right)'\right]_+\dt
				\overset{\eqref{eq::integrable_little0_1}}{=}  (\alpha+1)\int_{t_0}^{+\infty}t^\alpha E(t)\dt \overset{Ass.}{<} +\infty.
			\end{split}
		\end{equation}
		Therefore by definition of the improper integral, $ \displaystyle \lim_{t \to +\infty}t^{\alpha+1}E(t)$ admits a finite limit $l \in \mathbb{R}_{\geq 0} $.

		Now, assume that $ l>0 $, then for any $\varepsilon \in ]0,l[$ there exists $ t_1\geq t_0 $ such for all $ t \geq t_1 $ we have $ E(t) > \frac{l-\varepsilon}{t^{\alpha+1}} $.
		Then,
		\begin{equation*}
			\int_{t_0}^{+\infty}t^\alpha E(t)\dt \geq \int_{t_1}^{+\infty}t^\alpha E(t)\dt > \int_{t_1}^{+\infty} t^\alpha \frac{l-\varepsilon}{t^{\alpha +1}}\dt 
			= (l-\varepsilon) \int_{t_1}^{+\infty}\frac{1}{t}\dt = +\infty,
		\end{equation*}
		which contradicts the integrability assumption on $t^\alpha E(t)$.
	\end{proof}
	\begin{proof}[Proof of Lemma \ref{lem::littleo_integrable}]
		Let $l\eqdef \displaystyle \lim_{t \to \infty} t^\beta g(t)\geq0$. Then for any $ \varepsilon > 0 $ we can find $ t_1 \geq t_0 $ such that $\forall t \geq t_1$,
		\begin{equation*}\label{eq::littleo_integrable_1}
			g(t) \leq \frac{(l+\varepsilon)}{t^\beta} \ .
		\end{equation*}
		Therefore, $$ \int_{t_0}^{+\infty} g(t)\dt = \underset{\eqdef I< +\infty}{\underbrace{\int_{t_0}^{t_1}g(t)\dt}} + \int_{t_1}^{+\infty} g(t) \dt\leq
		I + (l+\varepsilon)\int_{t_1}^{+\infty}\frac{1}{t^\beta} \dt < +\infty,$$
		since $ \beta > 1 $, which proves the result.
	\end{proof}

	\FloatBarrier
	\bibliographystyle{plainnat}
	\bibliography{biblio.bib}
\end{document}